\documentclass[reqno, 12pt]{amsart}
\usepackage{amsmath, amsthm, amssymb, paralist, cite, color, bm}
\usepackage[all]{xypic}
\definecolor{e-mail}{rgb}{0,.40,.80}
\definecolor{reference}{rgb}{.20,.60,.22}
\definecolor{citation}{rgb}{0,.40,.80}
\usepackage[pdfstartview=FitH, colorlinks, linkcolor=reference, citecolor=citation, urlcolor=e-mail]{hyperref}

\newtheorem{thm}{Theorem}
\newtheorem{cor}[thm]{Corollary}
\newtheorem{lem}[thm]{Lemma}
\newtheorem{prop}[thm]{Proposition}

\theoremstyle{definition}
\newtheorem{defn}[thm]{Definition}
\theoremstyle{remark}
\newtheorem{rem}[thm]{Remark}
\numberwithin{thm}{section}
\theoremstyle{definition}

\theoremstyle{definition}

\theoremstyle{definition}
\numberwithin{equation}{section}

 \title[Differential Galois groups of $q$-difference equations]{Computing differential Galois groups of \\ second-order linear $q$-difference equations}
 
\author[C.E. Arreche]{Carlos E. Arreche}
\address{Department of Mathematical Sciences, The University of Texas at Dallas, Texas, USA}
\email{arreche@utdallas.edu}
\thanks{The work of both authors was partially supported by NSF grant CCF-1815108.}

\author[Y. Zhang]{Yi Zhang}
\address{Department of Applied Mathematics, Xi'an Jiaotong-Liverpool University, Suzhou, China}
\email{Yi.Zhang03@xjtlu.edu.cn}

\begin{document}


\begin{abstract}
We apply the differential Galois theory for difference equations developed by Hardouin and Singer to compute the differential Galois group for a second-order linear $q$-difference equation with rational function coefficients. This Galois group encodes the possible polynomial differential relations among the solutions of the equation. We apply our results to compute the differential Galois groups of several concrete $q$-difference equations, including for the colored Jones polynomial of a certain knot.
\end{abstract}

\maketitle


\section{Introduction} \label{intro-sec}

Consider a second-order homogeneous linear $q$-dilation equation \begin{equation}\label{intro-eq} y(q^2x)+a(x)y(qx)+b(x)y(x)=0,\end{equation} whose coefficients $a(x),b(x)\in\bar{\mathbb{Q}}(x)$ are rational functions in $x$ with $b(x)\neq 0$, and $q\in\bar{\mathbb{Q}}$ is neither zero nor a root of unity. We develop algorithms that allow one to discover all the polynomial differential equations satisfied by the solutions to \eqref{intro-eq}, or to decide that there are none. Our methods and results apply equally well, with small and obvious modifications, to equations \eqref{intro-eq} where $q$ is not necessarily an algebraic number and the coefficients $a,b\in C(x)$ for any computable algebraically closed field $C$ containing $\mathbb{Q}(q)$. Our strategy here is similar to the one followed in \cite{arreche:2017}, where analogous algorithmic results were developed in the context of shift difference equations. We apply the differential Galois theory for difference equations developed in \cite{hardouin-singer:2008}, which studies equations such as \eqref{intro-eq} from a purely algebraic point of view. This theory attaches a geometric object $G$ to \eqref{intro-eq}, called the differential Galois group, that encodes all the difference-differential algebraic relations among the solutions to \eqref{intro-eq}. We develop an algorithm to compute the differential Galois group $G$ associated to \eqref{intro-eq} by the theory of \cite{hardouin-singer:2008}.

The differential Galois theory for difference equations of \cite{hardouin-singer:2008} is a generalization of the $q$-dilation analogue of the Galois theory for difference equations presented in \cite{vanderput-singer:1997}, where the Galois groups that arise encode the algebraic relations among the solutions to a given linear difference equation. An algorithm to compute the Galois group $\tilde{H}$ associated to \eqref{intro-eq} by the theory of \cite{vanderput-singer:1997} is developed in \cite{hendriks:1997}---but for technical reasons this algorithm works only over the larger base field $\bar{\mathbb{Q}}(\{x^{1/n}\}_{n\in\mathbb{N}})$, rather than the field of definition $\bar{\mathbb{Q}}(x)$ of \eqref{intro-eq}. In the course of our computation of the differential Galois group $G$ of \eqref{intro-eq}, we also extend the algorithm of \cite{hendriks:1997} to compute the Galois group $H$ of \eqref{intro-eq} over the smaller original basefield $\bar{\mathbb{Q}}(x)$.

A priori one knows that the Galois group $H$ is a linear algebraic group, and the differential Galois group $G$ is a linear differential algebraic group (Definition~\ref{ldag-def}). The difference Galois group $H$ serves as a close upper bound for the difference-differential Galois group $G$: it is shown in \cite{hardouin-singer:2008} that one can consider $G$ as a Zariski-dense subgroup of $H$ without loss of generality (see Proposition~\ref{dense} for a precise statement). In view of this fact, our strategy to compute $G$ is to first apply our extension (developed in the present work) of the algorithm of \cite{hendriks:1997} to compute $H$, and then compute the additional differential-algebraic equations (if any) that define $G$ as a subgroup of $H$. The computation of $G$ in general can be much more difficult than that of $H$ because there are many more linear differential algebraic groups than there are linear algebraic groups (more precisely, the latter are instances of the former), so identifying the correct differential Galois group from among these additional possibilities requires additional work. 

This strategy is reminiscent of the one begun in \cite{dreyfus:2011}, and concluded in \cite{arreche:2014a, arreche:2014b, arreche:2015}, to compute the parameterized differential Galois group for a second-order linear differential equation with differential parameters, where the results of \cite{kovacic:1986, baldassari-dwork:1979} are first applied to compute the classical (non-parameterized) differential Galois group for the differential equation, and one then computes the additional differential-algebraic equations, with respect to the parametric derivations, that define the parameterized differential Galois group inside the classical one. However, the computation of the differential Galois group $G$ for \eqref{intro-eq} presents substantial new complications that do not arise in the parameterized differential setting. Many of these new complications are inherent to the computation of differential Galois groups of difference equations in general, and already arise in the context of shift difference equations (see the introduction to \cite{arreche:2017} for a summary), but a brand new technical difficulty arises for the first time in the context of $q$-difference equations, which we describe below. The same difficulties will recur, with a vengeance, in the context of Galois theory for difference equations over elliptic curves; our hope is that the treatment developed here will serve as a useful blueprint for that more technical setting.

It is known (see \cite{hendriks:1997}) that the Galois group $\tilde{H}$ of any $q$-difference equation over $\bar{\mathbb{Q}}(\{x^{1/n}\}_{n\in\mathbb{N}})$ has a cyclic group of connected components $\tilde{H}/\tilde{H}^\circ$. This fact facilitates the development of the algorithm of \cite{hendriks:1997}. However, the Galois group $H$ of a $q$-difference equation over $\bar{\mathbb{Q}}(x)$ may admit more generally a bicyclic group of connected components, which requires the development of new techniques to identify the correct Galois group from among this larger set of possibilities.

A theoretical consequence of the results of \S\ref{reducible-sec} is Corollary~\ref{unipotent-possibilities}, which states that the unipotent radical of the differential Galois group may only be trivial, the additive group of differentially constant $\sigma$-invariants, or the full additive group of $\sigma$-invariants. This result was already known when the whole differential Galois $G$ group was already unipotent \cite[Prop.~4.3(2)]{hardouin-singer:2008}, but not when the unipotent radical is a proper subgroup of $G$. In other contexts (see for example \cite{minchenko-ovchinnikov-singer:2013a, minchenko-ovchinnikov-singer:2013b}) the computation of the unipotent radical has turned out to be the main theoretical obstacle in the development of algorithms to compute Galois groups in general. We expect that this contribution to the inverse Galois problem in the present setting will have useful ramifications in the development of future algorithms to compute differential Galois groups for higher-order $q$-difference equations.

Let us now describe the contents of this work in more detail. In \S\ref{prelim-sec}, we summarize the difference-differential Galois theory of \cite{hardouin-singer:2008}, and prove some auxiliary results that will be used in the sequel. In \S\ref{order1-sec}, we recall some known results, and prove some new ones, concerning differential relations among solutions to first-order $q$-dilation difference equations. In \S\ref{hendriks-sec}, we summarize Hendriks' algorithm \cite{hendriks:1998} to compute the difference Galois group $\tilde{H}$ for \eqref{intro-eq} over $\bar{\mathbb{Q}}(\{x^{1/n}\}_{n\in\mathbb{N}})$, and explain how to extend it to compute the difference Galois group $H$ for \eqref{intro-eq} over $\bar{\mathbb{Q}}(x)$. In \S\ref{diagonalizable-sec}, we show how to compute the difference-differential Galois group $G$ for \eqref{intro-eq} when $H$ is diagonalizable in Proposition~\ref{galdiag}. In \S\ref{reducible-sec}, we show how to compute $G$ when $H$ is assumed to be reducible but non-diagonalizable in Proposition~\ref{wrat} and Proposition~\ref{wnotrat}---as a consequence, we show in Corollary~\ref{unipotent-possibilities} that the unipotent radical of $G$ is always of a very special form. In \S\ref{dihedral-sec}, we compute $G$ in Proposition~\ref{imprimitive-prop-1}, Proposition~\ref{imprimitive-prop-2}, and Proposition~\ref{imprimitive-prop-3}, under the assumption that $H$ is irreducible and imprimitive (which possibility can arise in three different ways, as a consequence of our insistence on computing Galois groups over the basefield $\bar{\mathbb{Q}}(x)$ and not just over $\bar{\mathbb{Q}}(\{x^{1/n}\}_{n\in\mathbb{N}})$). In \S\ref{large-sec}, we apply results from \cite{arreche-singer:2016} to compute $G$ in Proposition~\ref{sl2-prop}, under the assumption that $H$ contains $\mathrm{SL}_2$. We conclude in \S\ref{examples-sec} by applying these results to some concrete examples of $q$-difference equations; in particular to the one satisfied by the colored Jones polynomial of a certain knot.


\section{Preliminaries on differential Galois theory for difference equations} \label{prelim-sec}

We begin with a summary of the difference-differential Galois theory presented in \cite{hardouin-singer:2008}. Every field is assumed to be of characteristic zero, and every ring is assumed to be commutative unless otherwise stated.

\begin{defn} A $\sigma\delta$\emph{-ring} is a commutative ring $R$ with unit, equipped with an automorphism $\sigma$ and a derivation $\delta$ such that $\sigma\left(\delta(r)\right)=\delta\left(\sigma(r)\right)$ for every $r\in R$. A $\sigma\delta$-field is defined analogously. We write $$R^\sigma=\{r\in R \ | \ \sigma(r)=r\};\quad R^\delta=\{r\in R \ | \ \delta(r)=0\}; \quad\text{and}\quad R^{\sigma\delta}=R^\sigma\cap R^\delta,$$ and refer to these as the subrings of $\sigma$\emph{-constants}, $\delta$\emph{-constants}, and $\sigma\delta$\emph{-constants}, respectively.

A $\sigma\delta$-$R$-algebra is a $\sigma\delta$-ring $S$ equipped with a ring homomorphism $R\rightarrow S$ that commutes with both $\sigma$ and $\delta$. If $R$ and $S$ are fields, we also say that $S$ is a $\sigma\delta$-field extension of $R$. The notions of $\sigma$-$R$-algebra, $\delta$-$R$-algebra, $\sigma$-field extension, and $\delta$-field extension are defined analogously. If $z_1,\dots,z_n\in S$, we write $R\{z_1,\dots,z_n\}_\delta$ for the smallest $\delta$-$R$-subalgebra of $S$ that contains $z_1,\dots,z_n$; as $R$-algebras, we have $$R\{z_1,\dots,z_n\}_\delta=R[\{\delta^i(z_1),\dots,\delta^i(z_n) \ | \ i\in\mathbb{N}\}].$$ If $Z=(z_{ij})$ with $1\leq i,j\leq n$ is a matrix, we write $R=\{Z\}_\delta$ for \[R\{z_{11},\dots,z_{1n},\dots,z_{n1},\dots,z_{nn}\}_\delta.\]
\end{defn}

The main example of $\sigma\delta$-field that we will consider throughout most of this paper is $k=\bar{\mathbb{Q}}(x)$, where $\sigma$ denotes the $\bar{\mathbb{Q}}$-linear automorphism defined by $\sigma(x)= qx$ for some fixed $q\in\bar{\mathbb{Q}}$ that is neither zero nor a root of unity, and $\delta=x\tfrac{d}{dx}$. Note that in this case $k^\sigma=k^\delta=\bar{\mathbb{Q}}$.

Suppose that $k$ is a $\sigma\delta$-field, and consider the matrix difference equation \begin{equation}\label{difeq1} \sigma(Y)=AY, \quad \text{where} \ A\in\mathrm{GL}_n(k).\end{equation}

\begin{defn} \label{pv-ring-def}A $\sigma\delta$\emph{-Picard-Vessiot ring} (or $\sigma\delta$\emph{-PV ring}) over $k$ for \eqref{difeq1} is a $\sigma\delta$-$k$-algebra $R$ such that:
\begin{enumerate}
\item[(i)] $R$ is a \emph{simple} $\sigma\delta$-ring, i.e., $R$ has no ideals, other than ${0}$ and $R$, that are stable under both $\sigma$ and $\delta$;
\item[(ii)] there exists a matrix $Z\in\mathrm{GL}_n(R)$ such that $\sigma(Z)=AZ$; and
\item[(iii)] $R$ is differentially generated as a $\delta$-$k$-algebra by the entries of $Z$ and $1/\mathrm{det}(Z)$, i.e., $R=k\{Z,1/\mathrm{det}(Z)\}_\delta$.
\end{enumerate}
The matrix $Z$ is called a \emph{fundamental solution matrix} for \eqref{difeq1}. \end{defn}

Note that when $\delta=0$, this coincides with the definition of the $\sigma$-PV ring over $k$ for \eqref{difeq1} given in \cite[Def.~1.5]{vanderput-singer:1997}. In the usual Galois theory of difference equations presented in \cite{vanderput-singer:1997}, the existence and uniqueness of Picard-Vessiot rings up to $k$-$\sigma$-isomorphism is guaranteed by the assumption that $k^\sigma$ is algebraically closed (see \cite[\S1.1]{vanderput-singer:1997}). Analogously, in the difference-differential Galois theory developed in \cite{hardouin-singer:2008}, one needs to assume that $k^\sigma$ is $\delta$\emph{-closed} \cite{kolchin:1974, trushin:2010}.

\begin{defn} The ring of $\delta$-\emph{polynomials} in $n$ variables over a $\delta$-field $C$ is $$C\{Y_1,\dots,Y_n\}_\delta=C[\{\delta^i(Y_1),\dots,\delta^i(Y_n) \ | \ i\in\mathbb{N}\}],$$ the free $C$-algebra on the symbols $\delta^i(Y_j)$, on which $\delta$ acts as a derivation in the obvious way. We say $\mathcal{L}\in C\{Y_1,\dots,Y_n\}_\delta$ is a \emph{linear} $\delta$-\emph{polynomial} if it belongs to the $C$-linear span of the symbols $\delta^i(Y_j)$.

If $R$ is a $\delta$-$C$-algebra, we say that $z_1\dots,z_n\in R$ are \emph{differentially dependent} over $C$ if there exists a $\delta$-polynomial $0\neq P\in C\{Y_1,\dots,Y_n\}_\delta$ such that $P(z_1,\dots,z_n)=0$; otherwise we say that $z_1,\dots,z_n$ are $\delta$-\emph{independent} over $C$. When a single element $z\in R$ is $\delta$-independent (resp., $\delta$-dependent) over $C$, we say that $z$ is $\delta$-\emph{transcendental} (resp., $\delta$-\emph{algebraic}) over $C$.

We say the $\delta$-field $C$ is $\delta$-\emph{closed} if any system of $\delta$-polynomial equations $$\{P_1=0,\dots,P_m=0 \ | \ P_i\in C\{Y_1,\dots,Y_n\}_\delta \ \ \text{for} \ \ 1\leq i\leq m\}$$ that has a solution in $\tilde{C}^n$ for some $\delta$-field extension $\tilde{C}\supseteq C$ already has a solution in $C^n$.
\end{defn}

\begin{thm} \label{nnc}(Cf.~\cite[Prop.~2.4]{hardouin-singer:2008}) If $k^\sigma=C$ is $\delta$-closed, there exists a $\sigma\delta$-PV ring for \eqref{difeq1}, and it is unique up to $\sigma\delta$-$k$-isomorphism. Moreover, $R^\sigma=k^\sigma$.
\end{thm}

The following structural result is stated in a more general context in \cite[Lem.~6.8]{hardouin-singer:2008}, with the exception of the second part of item~(3), which is proved as in \cite[Cor.~1.16]{vanderput-singer:1997}.

\begin{prop} \label{idemp} Let $R$ be a $\sigma\delta$-PV ring over $k$ for \eqref{difeq1}, where $k^\sigma$ is $\delta$-closed. There exist idempotents $e_0,\dots,e_{t-1}\in R$ such that:
\begin{enumerate}
\item $R=R_0\oplus\dots\oplus R_{t-1}$, where $R_i= e_i R$;
\item the action of $\sigma$ permutes the set $\{ R_0,\dots,R_{t-1}\}$ transitively, and each $R_i$ is left invariant by $\sigma^t$;
\item each $R_i$ is a domain, and a $\sigma^t\delta$-PV ring over $k$ for $\sigma^t(Y)=(\sigma^{t-1}(A)\dots\sigma(A)A)Y$.
\end{enumerate}
\end{prop}

From now on, unless explicitly stated otherwise, we assume that $k$ is a $\sigma\delta$-field such that $k^\sigma$ is $\delta$-closed.

\begin{defn} The $\sigma\delta$\emph{-Galois group} of \eqref{difeq1} is the group of $\sigma\delta$-$k$-automorphisms of the $\sigma\delta$-PV ring $R$ for \eqref{difeq1}: $$\mathrm{Gal}_{\sigma\delta}(R/k)=\{\gamma\in\mathrm{Aut}_{k\text{-alg}}(R) \ | \ \gamma\circ\sigma = \sigma\circ\gamma \ \text{and} \ \gamma\circ\delta = \delta\circ\gamma\}.$$\end{defn}

As in the usual (non-differential) Galois theory of difference equations \cite{vanderput-singer:1997}, a choice of fundamental solution $Z=(z_{ij})\in\mathrm{GL}_n(R)$ defines a faithful representation $\mathrm{Gal}_{\sigma\delta}(R/k)\hookrightarrow\mathrm{GL}_n(k^\sigma):\gamma\mapsto M_\gamma$, via $$\gamma(Z)=\begin{pmatrix} \gamma(z_{11}) & \cdots & \gamma(z_{1n}) \\ \vdots & & \vdots \\ \gamma(z_{n1}) & \cdots & \gamma(z_{nn}) \end{pmatrix} = \begin{pmatrix} z_{11} & \cdots & z_{1n} \\ \vdots & & \vdots \\ z_{n1} & \cdots & z_{nn} \end{pmatrix}\cdot M_\gamma.$$ A different choice of fundamental solution matrix $Z'\in\mathrm{GL}_n(R)$ defines a conjugate representation of $\mathrm{Gal}_{\sigma\delta}(R/k)$ in $\mathrm{GL}_n(k^\sigma)$.

\begin{defn} \label{equivalent-def} The systems $\sigma(Y)=AY$ and $\sigma(Y)=BY$ for $A,B\in\mathrm{GL}_n(k)$ are \emph{equivalent} if there exists a matrix $T\in\mathrm{GL}_n(k)$ such that $\sigma(T)AT^{-1}=B$. In this case, if $Z$ is a fundamental solution matrix for $\sigma(Y)=AY$, then $TZ$ is a fundamental solution matrix for $\sigma(Y)=BY$, and therefore the $\sigma\delta$-PV rings of $k$ for these systems defined by the choice of fundamental solution matrices $Z$ and $TZ$, and the associated representations of $\sigma\delta$-Galois groups in $\mathrm{GL}_n(k^\sigma)$, are isomorphic.\end{defn}

\begin{defn} \label{ldag-def} Suppose that $C$ is a $\delta$-closed field. A \emph{linear differential algebraic group} over $C$ is a subgroup $G$ of $\mathrm{GL}_n(C)$ defined by (finitely many) $\delta$-polynomial equations in the matrix entries. We say that $G$ is $\delta$\emph{-constant} if $G$ is conjugate in $\mathrm{GL}_n(C)$ to a subgroup of $\mathrm{GL}_n(C^\delta)$.\end{defn}

The differential algebraic subgroups of the additive and multiplicative groups of $C$, which we denote respectively by $\mathbb{G}_a(C)$ and $\mathbb{G}_m(C)$, were classified in \cite[Prop.~11, Prop.~31 and its Corollary]{cassidy:1972}.

\begin{prop} \label{classification} If $G\leq \mathbb{G}_a(C)$ is a differential algebraic subgroup, then there exists a linear $\delta$-polynomial $\mathcal{L}\in C\{Y\}_\delta$ such that $$G=\{b\in\mathbb{G}_a(C) \ | \ \mathcal{L}(b)=0\}.$$

If $G\leq \mathbb{G}_m(C)$ is a differential algebraic subgroup, then either $G=\mu_\ell$, the group of $\ell^\text{th}$ roots of unity for some $\ell\in\mathbb{N}$, or else $\mathbb{G}_m(C^\delta)\subseteq G$, and there exists a linear $\delta$-polynomial $\mathcal{L}\in C\{Y\}_\delta$ such that $$G=\left\{a\in\mathbb{G}_m(C) \ \middle| \ \mathcal{L}\bigl(\tfrac{\delta a}{a}\bigr)=0\right\}.$$\end{prop}

\begin{thm} (Cf.~\cite[Thm.~2.6]{hardouin-singer:2008}) Suppose that $k^\sigma$ is $\delta$-closed, and that $R$ is a $\sigma\delta$-PV ring over $k$ for \eqref{difeq1}. Then $R$ is a reduced ring, and any choice of fundamental solution matrix $Z\in\mathrm{GL}_n(R)$ identifies $\mathrm{Gal}_{\sigma\delta}(R/k)$ with a linear differential algebraic subgroup of $\mathrm{GL}_n(k^\sigma)$.\end{thm}

As in \cite[p.~337]{hardouin-singer:2008}, we observe that if $R$ is a $\sigma\delta$-PV ring over $k$ for \eqref{difeq1}, and $K$ is the total ring of fractions of $R$, then any $\sigma\delta$-$k$-automorphism of $K$ must leave $R$ invariant, whence the group $\mathrm{Gal}_{\sigma\delta}(K/k)$ of such automorphisms coincides with $\mathrm{Gal}_{\sigma\delta}(R/k)$. The consideration of the total ring of fractions of $R$ is necessary to obtain the following Galois correspondence.

\begin{thm} \label{correspondence} (Cf.~\cite[Thm.~2.7]{hardouin-singer:2008}) Suppose that $k^\sigma$ is $\delta$-closed, and that $R$ is a $\sigma\delta$-PV ring over $k$ for \eqref{difeq1}. Denote by $K$ the total ring of fractions of $R$, and by $\mathcal{F}$ the set of $\sigma\delta$-rings $F$ such that $k\subseteq F \subseteq K$ and every non-zero divisor in $F$ is a unit in $F$. Let $\mathcal{G}$ denote the set of linear differential algebraic subgroups $G$ of $\mathrm{Gal}_{\sigma\delta}(K/k)$. There is a bijective correspondence $\mathcal{F}\leftrightarrow\mathcal{G}$ given by \begin{gather*}F \mapsto \mathrm{Gal}_{\sigma\delta}(K/F)=\{\gamma\in\mathrm{Gal}_{\sigma\delta}(K/k) \ | \ \gamma(r)=r, \ \forall r\in F\}; \quad \text{and} \\ G \mapsto K^G=\{r\in K \ | \ \gamma(r)=r, \ \forall \gamma\in G\}.\end{gather*}\end{thm}

\noindent In particular, an element $r\in K$ is left fixed by all of $\mathrm{Gal}_{\sigma\delta}(K/k)$ if and only if $r\in k$.

The following result relates the $\sigma\delta$-PV rings and $\sigma\delta$-Galois groups of \cite{hardouin-singer:2008} to the $\sigma$-PV rings and $\sigma$-Galois groups considered in \cite{vanderput-singer:1997, hendriks:1998}.

\begin{prop} \label{dense} (Cf.~\cite[Prop.~2.8]{hardouin-singer:2008}) Assume $k^\sigma$ is $\delta$-closed. Let $R$ be a $\sigma\delta$-PV ring over $k$ for \eqref{difeq1} with fundamental solution matrix $Z\in\mathrm{GL}_n(R)$, and let $S=k[Z,1/\mathrm{det}(Z)]\subset R$. Then:
\begin{enumerate}
\item[(i)] $S$ is a $\sigma$-PV ring over $k$ for \eqref{difeq1}; and
\item[(ii)] $\mathrm{Gal}_{\sigma\delta}(R/k)$ is Zariski-dense in the $\sigma$-Galois group $\mathrm{Gal}_\sigma(S/k)$.
\end{enumerate}
\end{prop}

The following result characterizes those difference equations whose $\sigma\delta$-Galois groups are $\delta$-constant.

\begin{prop} \label{dconst} (Cf.~\cite[Prop.~2.9]{hardouin-singer:2008}) Let $R$ be a $\sigma\delta$-PV ring over $k$ for $\sigma(Y)=AY$, where $A\in\mathrm{GL}_n(k)$ and $k^\sigma$ is $\delta$-closed. Then $\mathrm{Gal}_{\sigma\delta}(R/k)$ is a $\delta$-constant linear differential algebraic group if and only if there exists a matrix $B\in\mathfrak{gl}_n(k)$ such that $$\sigma(B)=ABA^{-1} +\delta(A)A^{-1}.$$ In this case, there exists a fundamental solution matrix $Z\in\mathrm{GL}_n(R)$ that satisfies the system $$ \begin{cases} \sigma(Z)  =AZ; \\ \delta(Z) = BZ.\end{cases}$$
\end{prop}


\section{Differential relations among solutions of first-order $q$-difference equations}\label{order1-sec}

In this section we recall some known results, and prove some new ones, concerning differential relations among solutions of first-order $q$-difference difference equations. The following result is proved in \cite[Prop.~3.1]{hardouin-singer:2008}.

\begin{prop}\label{diffsum} Let $R$ be a $\sigma\delta$-$k$-algebra with $R^\sigma=k^\sigma$. Suppose $b_1,\dots,b_m\in k$ and $z_1,\dots,z_m\in R$ satisfy $$\sigma(z_i)-z_i=b_i; \quad i=1,\dots,m.$$ Then $z_1,\dots,z_m$ are differentially dependent over $k$ if and only if there exists a nonzero linear $\delta$-polynomial $\mathcal{L}(Y_1,\dots,Y_m)$ with coefficients in $k^\sigma$ and an element $f\in k$ such that $$\mathcal{L}(b_1,\dots,b_m)=\sigma(f)-f.$$\end{prop}

For the remainder of this section, we restrict our attention to the $\sigma\delta$-field $k=C(x)$, where $\delta(x)=x$, $C$ is a $\delta$-closed field of characteristic zero, and $\sigma$ is the $C$-linear automorphism of $k$ defined by setting $\sigma(x)=qx$ for some $q\in C^\delta$ that is neither zero nor a root of unity.

The following notion of $q$-discrete residue, defined in \cite[Def.~2.7]{chen-singer:2012},will be crucial in several proofs in this paper.

\begin{defn} \label{DEF:qres} For any non-zero $\beta\in C$, we call the subset \[[\beta]_q=\beta q^\mathbb{Z}=\{\beta q^\ell \ | \ell\in\mathbb{Z}\}\subset  C\] the $q^\mathbb{Z}$-orbit of $\beta$ in $C$. Any $f\in k$ can be decomposed into the form \[f=c+xp_1+\frac{p_2}{x^s}+\sum_{i=1}^m\sum_{j=1}^{n_i}\sum_{\ell=0}^{d_{i,j}}\frac{\alpha_{i,j,\ell}}{(x-\beta_iq^\ell)^j},\] where $p_1,p_2\in C[x]$; $s,m,n_i,d_{i,j}\in\mathbb{N}$; $c,\alpha_{i,j,\ell},\beta_i\in C$; $\mathrm{deg}(p_2)<s$; and the $\beta_i$ are non-zero and belong to distinct $q^\mathbb{Z}$-orbits.

The $q$-\emph{discrete residue} of $f$ at the $q^\mathbb{Z}$-orbit $[\beta_i]_q$ of multiplicity $j$ (with respect to $x$) is defined as: \[q\text{-}\mathrm{dres} (f,[\beta_i]_q,j)=\sum_{\ell=0}^{d_{i,j}}q^{-\ell j}\alpha_{i,j,\ell}.\]

In addition, the constant $c$ above is the $q$-\emph{discrete residue} of $f$ at infinity, which we denote by $q$-$\mathrm{dres} (f,\infty)$.
\end{defn}

The usefulness of the notion of discrete residue stems from the following result.

\begin{prop}\label{indsum}(Cf.~\cite[Prop.~2.10]{chen-singer:2012})

Let $f,g\in C[x]$ be non-zero, relatively prime polynomials. There exists $h\in k$ such that $\sigma(h)-h=f/g$ if and only if $q$-$\mathrm{dres} (f/g,\infty)=0$ and $q$-$\mathrm{dres} (f/g,[\beta]_q,j)=0$ for every $j\in\mathbb{N}$ and every $0\neq \beta\in C$ such that $g(\beta)=0$.
\end{prop}

The following computational lemma will be used to sharpen the conclusion of \cite[Cor.~3.3]{hardouin-singer:2008} in the following Corollary~\ref{diffprod}.

\begin{lem} \label{qres-computation} Suppose $0\neq a\in C^\delta(x)$, $r\in\mathbb{Z}_{\geq 0}$, and $0\neq\beta\in C^\delta$ is a zero or pole of $a$. Then \[q\text{-}\mathrm{dres} \left(\delta^r\left(\frac{\delta(a)}{a}\right),[\beta]_q,r+1\right)=(-1)^r\cdot r!\cdot\beta^r\cdot q\text{-}\mathrm{dres} \left(\frac{\delta(a)}{a},[\beta]_q,1\right).\]
\end{lem}

\begin{proof} We may assume without loss of generality that 
\begin{equation}\label{qres-comp1}
\frac{\delta(a)}{a}=\sum_{\ell=0}^d\left(e_\ell + \frac{e_\ell q^\ell\beta}{x-\beta q^\ell}\right)\end{equation}
for some $0\neq\beta\in C^\delta$, $d\in\mathbb{Z}_{\geq 0}$, and $e_\ell\in\mathbb{Z}$ for $\ell=0,\dots,d$. Observe that in this case $q\text{-}\mathrm{dres} \bigl(\frac{\delta(a)}{a},[\beta]_q,1\bigr)=\sum_{\ell=0}^d\beta e_\ell$, by Definition~\ref{DEF:qres}. We claim that
\begin{equation}\label{qres-comp2}\delta^r\left(\frac{\delta(a)}{a}\right)=\sum_{\ell=0}^d\frac{(-1)^rr!q^{\ell (r+1)}\beta^{r+1}e_\ell}{(x-\beta q^\ell)^{r+1}} + (\text{lower-order terms}),\end{equation}
which would indeed imply that \[q\text{-}\mathrm{dres} \left(\delta^r\left(\frac{\delta(a)}{a}\right),[\beta]_q,r+1\right)=\sum_{\ell=0}^d(-1)^rr!\beta^{r+1}e_\ell\] and conclude the proof of the Lemma. We prove \eqref{qres-comp2} by induction. The case $r=0$ is just \eqref{qres-comp1}. Assuming \eqref{qres-comp2} for some $r\geq 0$, note that \[\delta^{r+1}\left(\frac{\delta(a)}{a}\right)=\sum_{\ell=0}^d\frac{(-1)^{r+1}(r+1)!q^{\ell(r+1)}\beta^{r+1}e_\ell\bigl((x-\beta q^\ell) + \beta q^\ell\bigr)}{(x-\beta q^\ell)^{r+2}} + (\text{lower-order terms}).\] This concludes the proof of the claim, and the Lemma.\end{proof}

The following result sharpens the conclusion of \cite[Cor.~3.3]{hardouin-singer:2008}.

\begin{cor} \label{diffprod}

Let $R$ be a $\sigma\delta$-$k$-algebra with $R^\sigma=k^\sigma=C$. Let $a_1,\dots,a_m\in C^\delta(x)^\times$ and $z_1,\dots,z_m\in R^\times$ such that $$\sigma(z_i)=a_iz_i; \quad i=1,\dots,m.$$ Then $z_1,\dots,z_m$ are differentially dependent over $k$ if and only if there exist: $n_1,\dots,n_m\in\mathbb{Z}$, not all zero and with $\mathrm{gcd}(n_1,\dots,n_m)=1$; $c \in \mathbb{Z}$; and an element $f\in k$, such that \begin{equation} \label{diffprod-eq}n_1\frac{\delta (a_1)}{a_1}+\dots+n_m\frac{\delta (a_m)}{a_m}  =\sigma(f)-f+c.\end{equation}\end{cor}

\begin{proof} First suppose there exist integers $n_1,\dots, n_m,c\in\mathbb{Z}$ as in \eqref{diffprod-eq}. Since for each $i=1,\dots,m$ we have that $\sigma\bigl(\frac{\delta(z_i)}{z_i}\bigr) =\frac{\delta(z_i)}{z_i} + \frac{\delta(a_i)}{a_i}$, it follows that \begin{multline*}\sigma\left[\left(\sum_{i=1}^mn_i\delta\left(\frac{\delta(z_i)}{z_i}\right)\right)-\delta(f)\right]=\sum_{i=1}^mn_i\delta\left(\frac{\delta(z_i)}{z_i}\right)+\delta\left(\sum_{i=1}^mn_i\frac{\delta(a_i)}{a_i}\right)-\sigma\bigl(\delta(f)\bigr)=\\
=\sum_{i=1}^mn_i\delta\left(\frac{\delta(z_i)}{z_i}\right)+\delta\bigl(\sigma(f)-f+c\bigr)-\sigma\bigl(\delta(f)\bigr)=\sum_{i=1}^mn_i\delta\left(\frac{\delta(z_i)}{z_i}\right)-\delta(f).\end{multline*} Therefore, \[\sum_{i=1}^mn_i\delta\left(\frac{\delta(z_i)}{z_i}\right)=\delta(f)+e\] for some $e\in R^\sigma=k^\sigma$. This shows that $z_1,\dots,z_m$ are $\delta$-dependent over $k$, after multiplying by $(z_1\dots z_m)^2$ on both sides.

Since $\sigma(\frac{\delta(z_i)}{z_i})=\frac{\delta(z_i)}{z_i}+\frac{\delta(a_i)}{a_i}$ for each $i=1,\dots,m$, Proposition~\ref{diffsum} implies that the $z_i$ are differentially dependent over $k$ if and only if there exists an element $f\in k$ and a nonzero linear $\delta$-polynomial $$\mathcal{L}(Y_1,\dots,Y_m)=\sum_{i=1}^m\sum_{j=0}^{r_i} c_{i,j}\delta^jY_i, \quad c_{i,j}\in C,$$ such that \begin{equation}\label{ord1rel}g=\mathcal{L}\left(\frac{\delta(a_1)}{a_1},\dots,\frac{\delta(a_m)}{a_m}\right)=\sigma(f)-f.\end{equation}

Let $r=\mathrm{max}\{\ r_i \ | \ c_{i,r_i}\neq 0 \ \text{for some} \ i\}$. For each $0\neq \beta\in C$, it follows from \eqref{ord1rel}, Proposition~\ref{indsum}, and Lemma~\ref{qres-computation}, that 
\begin{equation} \label{EQ:qres1}
q\text{-}\mathrm{dres} (g,[\beta]_q,r+1)=(-1)^r \cdot r! \cdot \beta^r \cdot  \sum_{i=1}^m c_{i,r} \cdot q\text{-}\mathrm{dres} \left(\frac{\delta(a_i)}{a_i},[\beta]_q,1\right)=0.
\end{equation}
On the other hand, it follows from Definition~\ref{DEF:qres} that for each $i=1,\dots,m$ we have that
\begin{equation} \label{EQ:qres2}
q\text{-}\mathrm{dres} \left(\frac{\delta(a_i)}{a_i},[\beta]_q,1\right) = \beta \cdot e_i \qquad \text{ for some } \qquad e_i \in \mathbb{Z}. 
\end{equation}
Substituting~\eqref{EQ:qres2} into~\eqref{EQ:qres1}, we have 
\begin{equation} \label{EQ:qres3}
q\text{-}\mathrm{dres} (g,[\beta]_q,r+1)=(-1)^r\cdot r! \cdot \beta^{r + 1} \cdot \sum_{i=1}^m c_{i,r} \cdot e_i = 0.
\end{equation}
Since $\beta \neq 0$, the above equation is equivalent to $\sum_{i=1}^m c_{i,r} \cdot e_i = 0$. 
Since $e_i \in \mathbb{Z}$ for each $i = 1, \ldots, m$, we may take the $c_{i,r}=n_i$ to be integers. 
Set $c =  \sum_{i = 1}^m n_i \cdot q\text{-}\mathrm{dres} \left(\frac{\delta(a_i)}{a_i}, \infty \right)$.
Since both $n_i$ and $q\text{-}\mathrm{dres} \left(\frac{\delta(a_i)}{a_i}, \infty \right)$ are integers, we see that $c \in \mathbb{Z}$ is divisible by $\mathrm{gcd}(n_1,\dots,n_m)$. 
Moreover, we have 
\begin{equation} \label{EQ:qres4}
q\text{-}\mathrm{dres} \left(n_1\frac{\delta (a_1)}{a_1}+\dots+n_m\frac{\delta (a_m)}{a_m} - c ,\infty\right) = 0.
\end{equation}
By~\eqref{EQ:qres3} and~\eqref{EQ:qres4}, the conclusion follows from another application of Proposition~\ref{indsum} and dividing both sides by $\mathrm{gcd}(n_1,\dots,n_m)$.
\end{proof}


\section{Hendriks' algorithm} \label{hendriks-sec}

In this section, we summarize the results of \cite{hendriks:1997} that we will need in our algorithm, and explain how to refine them to meet our goals. From now on, we restrict our attention to equations of the form \begin{equation}\label{difeq} \sigma^2(y) + a\sigma(y) + by = 0,\end{equation} where $a,b\in\bar{\mathbb{Q}}(x)$ with $b\neq 0$, and $\sigma$ is the $\bar{\mathbb{Q}}$-linear automorphism of $\bar{\mathbb{Q}}(x)$ defined by $\sigma(x)=qx$, where $q\in\bar{\mathbb{Q}}$ is neither zero nor a root of unity. Our discussion here could be generalized to drop the assumption that $q$ is an algebraic number and allowing $a,b\in C_0(x)$, for any computable algebraically closed field $C_0$ containing $\mathbb{Q}(q)$, as we mentioned in the introduction (see also the introduction to \cite{hendriks:1997}), but at the cost of overburdening the notation.

The matrix equation corresponding to \eqref{difeq} is \begin{equation}\label{mateq} \sigma(Y)=AY,\qquad \text{where}\quad A:=\begin{pmatrix} 0 & 1 \\ -b & -a\end{pmatrix}\in\mathrm{GL}_2(k).\end{equation}We consider $\bar{\mathbb{Q}}(x)$ as a $\sigma\delta$-field by setting $\delta=x\tfrac{d}{dx}$, the Euler derivation. In this section only we will denote $k=\bar{\mathbb{Q}}(x)$, but in future sections we will recycle notation and denote by $k$ the larger $\sigma\delta$-field $C(x)$, where $C$ is a $\delta$-closure of $\bar{\mathbb{Q}}$, and $\sigma$ is the $C$-linear automorphism of $C(x)$ defined by $\sigma(x)=qx$.

The algorithm of \cite{hendriks:1997} computes the $\sigma$-Galois group of \eqref{difeq} over the larger basefield $k_\infty$ defined as follows. Let $\{q_n \in\bar{\mathbb{Q}}\ |\ n\in\mathbb{N}\}$ denote a compatible system of $n$-th roots of $q=q_1$, so that for any factorization $\ell m=n$ we have $q_n^\ell=q_m$, and consider the cyclic $\sigma$-field extension $k_n=\bar{\mathbb{Q}}(x_n)$ of $\bar{\mathbb{Q}}(x)$ such that $x_n^n=x_1=x$ and $x_n^\ell=x_m$ for any factorization $n=\ell m$, with the $\sigma$-field structure given by $\sigma(x_n)=q_nx_n$. Then the $\bar{\mathbb{Q}}$-linear maps $k_m\hookrightarrow k_n$ defined by $x_m\mapsto x_n^\ell$ are embeddings of $\sigma$-fields. Let $k_\infty=\varinjlim k_n =\bigcup_{n\geq 1}k_n$. By \cite[Lemmas~9 and 10]{hendriks:1997}, the $\sigma$-field $k_\infty$ has property $\mathcal{P}$:

\begin{defn} \label{property_P-def} We say a $\sigma$-field $k$ has property $\mathcal{P}$ if:
\begin{enumerate}
\item $k$ is a $\mathcal{C}^1$ field; and
\item if $k'$ is a finite algebraic extension of $k$ such that $\sigma$ extends to an automorphism of $k'$ then $k'=k$.
\end{enumerate}
\end{defn}
\noindent This allows Hendriks to compute the $\sigma$-Galois group of \eqref{difeq} over $k_\infty$ by finding a gauge transformation $T\in\mathrm{GL}_2(k_\infty)$ that puts \eqref{mateq} in the standard form of \cite[Definition~8]{hendriks:1997}.

Another special consequence of the fact that $k_\infty$ enjoys property $\mathcal{P}$ (Definition~\ref{property_P-def}) is that the $\sigma$-Galois group $H_\infty$ for \eqref{difeq} over $k_\infty$ (and in fact every difference Galois group over $k_\infty$) is such that its quotient $H_\infty/H_\infty^\circ$ by the connected component of the identity $H_\infty^\circ$ must be a (finite) cyclic group (cf.~\cite[Thm.~6]{hendriks:1997}). This facilitates the algorithm of \cite{hendriks:1997} by ruling out a priori the consideration of algebraic groups whose group of connected components is not cyclic (cf~\cite[Lem.~12]{hendriks:1997}). The situation for $\sigma$-Galois groups over $k_1$ is less restrictive, but we still know by \cite[Prop.~12.2(1)]{vanderput-singer:1997} that the $\sigma$-Galois group $H_1$ for \eqref{difeq} over $k_1$ (and in fact every difference Galois group over $k_1$) has the property that the quotient $H_1/H_1^\circ$ is (finite) bicyclic, i.e., a product of two finite cyclic groups. Thus it is possible for us to realize additional algebraic groups $H_1$ as Galois groups for \eqref{difeq} over $k_1$ that do not occur in the list \cite[Lem.~16 and Lem.~20]{hendriks:1997} of possible $\sigma$-Galois groups over $k_\infty$. In particular, any reducible algebraic subgroup of $\mathrm{GL}_2(\bar{\mathbb{Q}})$ can (and does) occur as the $\sigma$-Galois group for some difference equation \eqref{difeq}, and any irreducible imprimitive algebraic subgroup of $\mathrm{GL}_2(\bar{\mathbb{Q}})$ with bicyclic group of connected components can (and does) occur as a Galois group over $k_1$. 

The algorithm developed in \cite{hendriks:1997} to compute $H_\infty$ proceeds as follows. We first decide whether there exists a solution $u\in k_\infty$ to the \emph{Riccati equation} \begin{equation}\label{ric1} u\sigma(u)+au+b=0.\end{equation} If such a solution $u$ exists, then the $\sigma$-Galois group $H_\infty$ of \eqref{mateq} over $k_\infty$ is \emph{reducible}, i.e., conjugate to an algebraic subgroup of $$\mathbb{G}_m(\bar{\mathbb{Q}})^2\ltimes\mathbb{G}_a(\bar{\mathbb{Q}})\simeq\left\{\begin{pmatrix} \alpha & \beta \\ 0 & \lambda \end{pmatrix} \ \middle| \ \alpha,\beta,\lambda\in \bar{\mathbb{Q}}, \ \alpha\lambda\neq 0\right\}.$$ Moreover, if there exist at least two distinct solutions $u_1, u_2\in k_\infty$ to \eqref{ric1} then $H_\infty$ is \emph{diagonalizable}, i.e., conjugate to an algebraic subgroup of $$\mathbb{G}_m(\bar{\mathbb{Q}})^2\simeq\left\{\begin{pmatrix} \alpha & 0 \\ 0 & \lambda\end{pmatrix} \ \middle| \ \alpha,\lambda\in \bar{\mathbb{Q}}, \ \alpha\lambda\neq 0\right\};$$ and if there are at least three distinct solutions in $k_\infty$ to \eqref{ric1} then there are infinitely many, and this occurs if and only if $H_\infty$ is an algebraic subgroup of $$\mathbb{G}_m(\bar{\mathbb{Q}})\simeq\left\{\begin{pmatrix} \alpha & 0 \\ 0 & \alpha\end{pmatrix} \ \middle| \ \alpha\in \bar{\mathbb{Q}}, \ \alpha\neq 0\right\}.$$

If there is no solution $u\in k_\infty$ to the Riccati equation \eqref{ric1}, then $H_\infty$ is irreducible by \cite[Thm.~13]{hendriks:1997}. In this case, the next step is to attempt to find $T\in\mathrm{GL}_2(k_\infty)$ and $r\in k_\infty$ such that \begin{equation}\label{impequiv}\sigma(T)\begin{pmatrix} 0 & 1 \\ -b & -a\end{pmatrix}T^{-1}=\begin{pmatrix} 0 & 1 \\ -r & 0\end{pmatrix}.\end{equation} If $a=0$ already, then we may take $T=\left(\begin{smallmatrix}1 & 0 \\ 0 & 1\end{smallmatrix}\right)$ and $r=b$. If $a\neq 0$, we then attempt to find a solution $e\in k_\infty$ to the Riccati equation \begin{equation}\label{ric2} e\sigma^2(e) + \bigl(\sigma^2(\tfrac{b}{a}) - \sigma(a) + \tfrac{\sigma(b)}{a}\bigr)e + \tfrac{\sigma(b)b}{a^2}=0.\end{equation} If there exists such a solution $e\in k_\infty$ to \eqref{ric2}, then it is proved in \cite[Thm.~18]{hendriks:1997} that there exists a matrix $T\in\mathrm{GL}_2(k_\infty)$ such that \eqref{impequiv} is satisfied with \begin{equation}\label{imprimitive-r}r=-a\sigma(a)+\sigma(b)+a\sigma^2(\tfrac{b}{a})+a\sigma^2(e),\end{equation} and $H_\infty$ is \emph{imprimitive}, i.e., conjugate to an algebraic subgroup of \begin{equation}\label{dihedral}\{\pm1\}\ltimes\mathbb{G}_m(\bar{\mathbb{Q}})^2\simeq\left\{\begin{pmatrix} \alpha & 0 \\ 0 & \lambda\end{pmatrix} \ \middle| \ \alpha,\lambda\in \bar{\mathbb{Q}}, \ \alpha\lambda\neq 0\right\} \cup\left\{\begin{pmatrix} 0 & \beta \\ \epsilon & 0\end{pmatrix} \ \middle| \ \beta,\epsilon\in \bar{\mathbb{Q}}, \ \beta\epsilon\neq 0\right\}.\end{equation}

Finally, if $a\neq 0$ and neither \eqref{ric1} nor \eqref{ric2} admits a solution in $k_\infty$, then $\mathrm{SL}_2(\bar{\mathbb{Q}})\subseteq H_\infty$, and we compute $H_\infty$ as in \cite[\S4.4]{hendriks:1997}, by determining the image $\mathrm{det}(H_\infty)\subseteq \mathbb{G}_m(\bar{\mathbb{Q}})$ of the determinant homomorphism.

In order to produce an algorithm that computes the $\sigma$-Galois group of \eqref{difeq} over $k=k_1$, we introduce additional notation and state some ancillary results. Let $\zeta_n\in\bar{\mathbb{Q}}$ for $n\in\mathbb{N}$ denote a compatible system of $n$-th roots of unity, so that for any factorization $\ell m=n$ we have $\zeta_n^\ell=\zeta_m$. Then $k_n$ is a $\sigma$-PV ring over $k_m$ for $\sigma(y)=q_ny$ with fundamental solution ($1\times 1$ matrix) $y=x_n$ and cyclic $\sigma$-Galois group $\langle\tau_{n,m}\rangle=\mathrm{Gal}_{\sigma}(k_n/k_m)$ given by $\tau_{n,m}(x_n)=\zeta_\ell x_n$. Let $S_\infty$ denote a $\sigma$-PV ring over $k_\infty$ for \eqref{mateq} with fundamental solution matrix $Y$. Then $S_n=k_n[Y,1/\mathrm{det}(Y)]$ is a $\sigma$-PV ring over $k_n$ for \eqref{mateq}. Let us write $H_n=\mathrm{Gal}_\sigma(S_n/k_n)$ for $n\in\mathbb{N}\cup\{\infty\}$. Then we see that $S_n$ is a $\sigma$-PV ring over $k_1$ for the system \[\sigma(Y_n)=\begin{pmatrix} 0 & 1 & 0 \\ -b & -a & 0 \\ 0 & 0 & q_n\end{pmatrix}Y_n, \ \ \text{with fundamental solution matrix}\ \ Y_n=\begin{pmatrix} y_1 & y_2 & 0\\ \sigma(y_1) & \sigma (y_2) & 0 \\ 0 & 0 & x_n\end{pmatrix}.\] The following result is proved formally as in~\cite[Lem.~3.1 and Prop.~3.2]{arreche:2014b} and \cite[Lem.~12 and Prop.~13]{arreche:2015}. Full proofs will appear in \cite{arreche:2020}.

\begin{prop} \label{fiber-product}Let $\tilde{H}_n:=\mathrm{Gal}_\sigma(S_n/k_1)$ and $\mu_n$ denote cyclic group of $n$-th roots of unity. Then the intersection $S_1\cap k_n=k_m$ for some factorization $n=\ell m$, and the map \begin{align*}\varphi:\tilde{H}_n &\rightarrow H_1\times \mu_n\\ \gamma &\mapsto (\gamma|_{S_1},\gamma|_{k_n})\end{align*} is an isomorphism onto the fiber product \begin{equation}\label{fiber-product-eq}H_1\times_{\mu_m}\mu_n=\{(\gamma,\zeta) \ | \ \gamma\in H_1 \ \text{and} \ \zeta\in\mu_n \ \text{such that} \ \gamma(x_m)=\zeta^\ell x_m\}.\end{equation} 
\end{prop}

We record the following two consequences of Proposition~\ref{fiber-product}.

\begin{cor} \label{fiber-cor-1} $S_1\cap k_n=k_1$ if and only if $H_n\simeq H_1$, and $S_1\cap k_n=k_n$ if and only if $H_n$ is a normal subgroup of $H_1$ of index $n$.
\end{cor}

\begin{cor}\label{fiber-cor-2} The intersection in $S_\infty$ given by $k_\infty\cap S_1=k_m$ for some $m\in\mathbb{N}$, and $H_\infty\simeq H_{\ell m}$ for every $\ell\in \mathbb{N}$. In particular, $H_\infty$ is a normal subgroup of $H_1$ of index $m$.
\end{cor}

Having computed the $\sigma$-Galois group $H_\infty$ of \eqref{difeq} over $k_\infty$ as in \cite{hendriks:1997}, we can then compute the $\sigma$-Galois group $H_1$ of \eqref{difeq} over $k_1$ according to the following possibilities. The explicit computation of $H_1$ is obtained in each case as a by-product of our computation of the corresponding differential Galois group of \eqref{difeq} in the following sections.

\begin{prop} \label{hendriks-possibilities}Precisely one of the following possibilities occurs.
\begin{enumerate}
\item There are infinitely many solutions to \eqref{ric1} in $k_1$. In this case, $H_1$ is a subgroup of $\mathbb{G}_m(\bar{\mathbb{Q}})$ (included in $\mathrm{GL}_2(\bar{\mathbb{Q}})$ as scalar matrices).
\item There are exactly two solutions $u_1,u_2\in k_1$ to \eqref{ric1}. In this case, $H_1$ is diagonalizable (but not contained in the group of scalar matrices).
\item There is exactly one solution $u\in k_1$ to \eqref{ric1}. In this case, $H_1$ is reducible but not diagonalizable.
\item There are no solutions to \eqref{ric1} in $k_1$, but there are exactly two solutions $u_1,u_2\in k_2\backslash k_1$ to \eqref{ric1}, and $u_2=\bar{u}_1$ is the Galois conjugate of $u_1$ over $k_1$. In this case, $H_1$ is irreducible and imprimitive.
\item There are no solutions to \eqref{ric1} in $k_2$, and either $a=0$ or there is a solution $e\in k_2$ to \eqref{ric2}. In this case, $H_1$ is irreducible and imprimitive.
\item There are no solutions to \eqref{ric1} nor to \eqref{ric2} in $k_2$ and $a\neq 0$. In this case, $H_1$ is irreducible and primitive, and $\mathrm{SL}_2(\bar{\mathbb{Q}})\subseteq H_1$.
\end{enumerate}
\end{prop}

\begin{proof} It is clear that the possibilities above are mutually exclusive. It remains to show that these possibilities are exhaustive, and that the $\sigma$-Galois group $H_1$ is as stated in each case.

Let us first show that these possibilities are exhaustive. By \cite[Thm.~13]{hendriks:1997}, there are either zero, one, two, or infinitely many solutions to \eqref{ric1} in $k_\infty$. By \cite[Thm.~15]{hendriks:1997}, if there exists a solution $u\in k_\infty$ to the Riccati equation \eqref{ric1}, then there exists a solution in $k_2$. Since the coefficients $a,b\in k=k_1$, for any solution $u\in k_2\backslash k_1$ to \eqref{ric1} the Galois conjugate $\bar{u}:=\tau_{2,1}(u)$ must also satisfy \eqref{ric1}. Hence, if there is exactly one solution $u\in k_\infty$ to \eqref{ric1}, then $u\in k_1$, and if there are exactly two solutions $u_1,u_2\in k_\infty$ to \eqref{ric1}, then either $u_1,u_2\in k_1$, or else $u_1,u_2\in k_2\backslash k_1$ and $u_2=\bar{u}_1$ is the Galois conjugate of $u_1$ over $k_1$. In the case where there are infinitely many solutions to \eqref{ric1} in $k_\infty$, the proof of \cite[Thm.~15]{hendriks:1997} shows that at least three of these solutions actually belong to $k_1$, in which case the proof of \cite[Thm.~4.2]{hendriks:1998} shows that there are infinitely many solutions to \eqref{ric1} in $k_1$. This shows that cases (1)--(4) exhaust the possibilities where there is at least one solution to \eqref{ric1} in $k_\infty$. Supposing now that there are no solutions to \eqref{ric1} in $k_\infty$ and $a\neq 0$, by \cite[Thm.~15]{hendriks:1997} we again have that if there exists at least one solution in $k_\infty$ to \eqref{ric2}, then there exists a solution in $k_2$. This concludes the proof that the possibilities listed in Proposition~\ref{hendriks-possibilities} are exhaustive and mutually exclusive.

The statements corresponding the form of the $\sigma$-Galois group $H_1$ will be established separetely in the following sections according to the possibilities listed above, depending on the existence of solutions to \eqref{ric1} or \eqref{ric2} in $k_1$ or $k_2$ as discussed above.\end{proof}

In view of Proposition~\ref{dense}, in order to compute the $\sigma\delta$-group $G$ of \eqref{difeq}, we will first apply the results of \cite{hendriks:1997} to compute the solutions to \eqref{ric1} and/or \eqref{ric2} in $k_2$, which according to the possibilities in Proposition~\ref{hendriks-possibilities} (and as we will show in each case in the following sections) results in knowing whether the corresponding $\sigma$-Galois group $H$ is: diagonalizable; reducible (but not diagonalizable); irreducible and imprimitive; or irreducible and primitive. We will then compute the additional $\delta$-algebraic equations that define $G$ as a subgroup of $H$ in each case (and obtain the explicit computation of $H$ itself along the way). In order to apply the theory of \cite{hardouin-singer:2008} to study \eqref{difeq}, we will consider \eqref{difeq} as a difference equation over the larger basefield $C(x)$ mentioned at the beginning of this section, where we recall $C$ is a $\delta$-closed field extension of $(\bar{\mathbb{Q}},\delta)$ such that $C^\delta=\bar{\mathbb{Q}}$ (the existence of such a $C$ is guaranteed by \cite{kolchin:1974, trushin:2010}), and the $\sigma\delta$-structure of $C(x)$ extends that of $\bar{\mathbb{Q}}(x)$: $\sigma$ is the $C$-linear automorphism of $C(x)$ defined by $\sigma(x)=qx$.

\begin{rem}[Descent from $C(x)$ to $\bar{\mathbb{Q}}(x)$] \label{c-ok} The application of the results of \cite{hendriks:1997} and Proposition~\ref{hendriks-possibilities} to compute the $\sigma$-Galois group of \eqref{difeq} over $C(x)$, rather than over $\bar{\mathbb{Q}}(x)$, requires some justification. The point is that the number of solutions to the Riccati equations \eqref{ric1} and \eqref{ric2} in $C(x_2)$ is the same as the number of solutions in $\bar{\mathbb{Q}}(x_2)$. This follows from an elementary argument: suppose that a given polynomial $\sigma$-equation over $\bar{\mathbb{Q}}(x)$ admits a solution $\frac{p}{q}\in C(x)$, where $p=a_nx^n+\dots+a_1x+a_0$ and $q=b_mx^m+\dots+b_1x+b_0$. This is equivalent to the coefficients $a_i$ and $b_j$ satisfying $b_m\neq 0$ and a system of polynomial equations defined over $\bar{\mathbb{Q}}$, which defines an affine algebraic variety $V$ over $\bar{\mathbb{Q}}$. Since $\bar{\mathbb{Q}}$ is algebraically closed and $C$ is countable, $V(C)$ and $V(\bar{\mathbb{Q}})$ must have the same cardinality.

Moreover, the possible defining equations for the $\sigma$-Galois groups of \eqref{difeq} over $\bar{\mathbb{Q}}(x)$ and over $C(x)$, whether in the reducible, irreducible and imprimitive, or irreducible and primitive cases, are all witnessed by monomial relations among (the standard form of) elements in $\bar{\mathbb{Q}}(x)$. Though we will see this explicitly in each situation in the following sections, it is worthwhile to emphasize now that the $\sigma$-Galois group of \eqref{difeq} over $C(x)$ consists of the $C$-points of the $\sigma$-Galois group over $\bar{\mathbb{Q}}(x)$, i.e., the former is defined as an algebraic subgroup of $\mathrm{GL}_2(C)$ by the same algebraic equations defining the latter as an algebraic subgroup of $\mathrm{GL}_2(\bar{\mathbb{Q}})$.\end{rem}


\section{Diagonalizable groups} \label{diagonalizable-sec} We recall the notation introduced in the previous sections: $k=C(x)$, $C$ is a $\delta$-closure of $\bar{\mathbb{Q}}$ with $C^\delta=\bar{\mathbb{Q}}$, $\sigma$ denotes the $C$-linear automorphism of $k$ defined by $\sigma(x)=qx$, and $\delta(x)=x$. Let us first suppose that there exist at least two distinct solutions $u_1,u_2\in\bar{\mathbb{Q}}(x)$ to the Riccati equation \eqref{ric1} as in items~(1) or (2) of Proposition~\ref{hendriks-possibilities}. Then \eqref{mateq} is equivalent over $\bar{\mathbb{Q}}(x)$ to $$\sigma(Y)=\begin{pmatrix} u_1 & 0 \\ 0 & u_2\end{pmatrix}Y,$$ in view of the following remark.

\begin{rem} \label{imprimitive-gauge} Given two distinct solutions $u_1$ and $u_2$ to \eqref{ric1}, the gauge transformation (which is different from the one specified in the proof of \cite[Thm.~4.2]{hendriks:1998}) \[T:=\frac{1}{u_1-u_2}\cdot\begin{pmatrix}u_2 & -1 \\ u_1 & -1\end{pmatrix}\] satisfies $\sigma(T)AT^{-1}=\left(\begin{smallmatrix} u_1 & 0 \\ 0 & u_2\end{smallmatrix}\right)$.\end{rem}

In this case, we compute $G$ with the following result.

\begin{prop} \label{galdiag} Assume that $u_1,u_2\in \bar{\mathbb{Q}}(x)$ are both different from $0$, and let $R$ be the $\sigma\delta$-PV ring over $k$ corresponding to the system\begin{equation}\label{diagonal-mateq}\sigma(Y)=\begin{pmatrix} u_1 & 0 \\ 0 & u_2\end{pmatrix}Y.\end{equation} Then $G=\mathrm{Gal}_{\sigma\delta}(R/k)$ is the subgroup of \begin{equation}\label{galdiag-eq}\mathbb{G}_m(C)^2=\left\{\begin{pmatrix} \alpha_1 & 0 \\ 0 & \alpha_2\end{pmatrix} \ \middle| \ \alpha_1,\alpha_2\in C, \ \alpha_1\alpha_2\neq 0 \right\}\end{equation} defined by the following conditions on $\alpha_1$ and $\alpha_2$.

\begin{enumerate}
\item[(i)] There exist $m_1,m_2\in\mathbb{Z}$, not both zero, and $f\in\bar{\mathbb{Q}}(x)^\times$ such that $u_1^{m_1}u_2^{m_2}=\frac{\sigma(f)}{f}$ if and only if $\alpha_1^{m_1}\alpha_2^{m_2}=1$.
\item[(ii)] There exist $m_1,m_2\in\mathbb{Z}$, not both zero and with $\mathrm{gcd}(m_1,m_2)=1$; $c\in\mathbb{Z}$; and $f\in \bar{\mathbb{Q}}(x)$ such that $m_1\frac{\delta(u_1)}{u_1}+m_2\frac{\delta(u_2)}{u_2}=\sigma(f)-f+c$ if and only if $\delta(m_1\frac{\delta(\alpha_1)}{\alpha_1}+m_2\frac{\delta(\alpha_2)}{\alpha_2})=0$. Moreover, $c=0$ if and only if $\delta(\alpha_1^{m_1}\alpha_2^{m_2})=0$.
\item[(iii)] If neither of the conditions above is satisfied, then $G=H=\mathbb{G}_m(C)^2$.
\end{enumerate}
\end{prop}

\begin{proof} We begin by observing that, if we can find $ f\in k$ witnessing the relations in items (i) or (ii), then we may take $f\in \bar{\mathbb{Q}}(x)$, since $u_i\in\bar{\mathbb{Q}}(x)$ (cf. \cite[Lem.~2.4, Lem.~2.5]{hardouin:2008} and Remark~\ref{c-ok}). Note that by Theorem~\ref{nnc}, $R^\sigma=C$. Let $y_1,y_2\in R$ be non-zero elements such that $\sigma(y_i)=u_iy_i$. Then $\left(\begin{smallmatrix}y_1 & 0 \\ 0 & y_2\end{smallmatrix}\right)$ is a fundamental solution matrix for \eqref{diagonal-mateq}, so $y_1,y_2\in R^\times$ and the embedding of $G$ into \eqref{galdiag-eq} is given by $\gamma(y_i)=\alpha_{\gamma,i}y_i$ for $i=1,2$ and $\gamma\in G$.

The proof of item~(i) is standard: given $m_1,m_2\in\mathbb{Z}$ we have that $\alpha_{\gamma,1}^{m_1}\alpha_{\gamma,2}^{m_2}=1$ for every $\gamma\in G$ if and only if $\gamma(y_1^{m_1}y_2^{m_2})=y_1^{m_1}y_2^{m_2}$ for every $\gamma \in G$. By Theorem~\ref{correspondence}, this is equivalent to $y_1^{m_1}y_2^{m_2}=f\in k$, which in turn is equivalent to $\frac{\sigma(f)}{f}=u_1^{m_1}u_2^{m_2}$.

Setting $\frac{\delta(y_i)}{y_i}=:g_i\in R$ for $i=1,2$, we see that \begin{equation}\label{diagonal-grels}\sigma(g_i)-g_i=\frac{\delta(u_i)}{u_i} \qquad \text{and}\qquad \gamma(g_i)=g_i+\frac{\delta(\alpha_{\gamma,i})}{\alpha_{\gamma,i}}.\end{equation} By Corollary~\ref{diffprod}, $y_1$ and $y_2$ are differentially dependent over $k$ if and only if there exist $m_1,m_2\in\mathbb{Z}$, not both zero and with $\mathrm{gcd}(m_1,m_2)=1$, $c\in\mathbb{Z}$, and $f\in k$ such that \begin{equation}\label{diagonal-urels}m_1\frac{\delta(u_1)}{u_1}+m_2\frac{\delta(u_2)}{u_2}=\sigma(f)-f+c.\end{equation} Hence, if there do not exist such $m_1,m_2,c\in\mathbb{Z}$ and $f\in k$, $y_1$ and $y_2$ are $\delta$-independent over $k$, which implies that $ G=\mathbb{G}_m(C)^2$ by \cite[Prop.~6.26]{hardouin-singer:2008}. This proves item~(iii).

Let us establish item~(ii). It follows from \eqref{diagonal-grels} that for any $m_1,m_2\in\mathbb{Z}$ we have \begin{gather*}\sigma(m_1g_1+m_2g_2)-(m_1g_1+m_2g_2)=m_1\frac{\delta(u_1)}{u_1}+m_2\frac{\delta(u_2)}{u_2};\qquad\text{and} \\ \gamma(m_1g_1+m_2g_2)=(m_1g_1+m_2g_2)+m_1\frac{\delta(\alpha_{\gamma,1})}{\alpha_{\gamma,1}}+ m_2\frac{\delta(\alpha_{\gamma,2})}{\alpha_{\gamma,2}}.\end{gather*}

Suppose there exists $f\in  k$ satisfying \eqref{diagonal-urels} with $c=0$ and $\mathrm{gcd}(m_1,m_2)=1$. Then $m_1g_1+m_2g_2-f\in k^\sigma$, which implies that $m_1g_1+m_2g_2\in k$ and therefore $\delta(\alpha_{\gamma,1}^{m_1}\alpha_{\gamma,2}^{m_2})=0$ for every $\gamma\in G$. On the other hand, if $\delta(\alpha_{\gamma,1}^{m_1}\alpha_{\gamma,2}^{m_2})=0$ for every $\gamma\in G$ with at least one $m_i\neq 0$, then the same relation holds after replacing $m_i$ with $\frac{m_i}{\mathrm{gcd}(m_1,m_2)}$  and we see that $m_1g_1+m_2g_2=f\in k$ satisfies \eqref{diagonal-urels} with $c=0$.

More generally, suppose there exist $f\in k$ and $c\in\mathbb{Z}$ satisfying \eqref{diagonal-urels} with $\mathrm{gcd}(m_1,m_2)=1$. Then we see that $m_1\delta(g_1)+m_2\delta(g_2)-\delta(f)\in k^\sigma$, and therefore $m_1\delta(g_1)+m_2\delta(g_2)\in k$, which implies that \begin{equation}\label{diagonal-arels}\delta\left(m_1\frac{\delta(\alpha_{\gamma,1})}{\alpha_{\gamma,1}}+ m_2\frac{\delta(\alpha_{\gamma,2})}{\alpha_{\gamma,2}}\right)=0\qquad \text{for every} \qquad \gamma\in G.\end{equation} On the other hand, assuming \eqref{diagonal-arels} with at least one $m_i\neq 0$, then the same relation holds after replacing $m_i$ with $\frac{m_i}{\mathrm{gcd}(m_1,m_2)}$, and we have that $m_1\delta(g_1)+m_2\delta(g_2)=g\in k$, and therefore \[m_1\delta\left(\frac{\delta(u_1)}{u_1}\right)+m_2\delta\left(\frac{\delta(u_2)}{u_2}\right)=\sigma(g)-g.\] By Proposition~\ref{indsum}, for each $\beta\in\mathbb{Q}^\times$ we have that \[0=q\text{-}\mathrm{dres} \left(\delta\left(\frac{\delta(u_1^{m_1}u_2^{m_2})}{u_1^{m_1}u_2^{m_2}}\right),[\beta]_q,2\right)=-\beta\cdot q\text{-}\mathrm{dres} \left(\frac{\delta(u_1^{m_1}u_2^{m_2})}{u_1^{m_1}u_2^{m_2}},[\beta]_q,1\right),\] where the second equality follows from Lemma~\ref{qres-computation}. Hence, letting \begin{equation}\label{diagonal-crels}c:=q\text{-}\mathrm{dres}\left(\frac{\delta(u_1^{m_1}u_2^{m_2})}{u_1^{m_1}u_2^{m_2}},\infty\right)=m_1\cdot q\text{-}\mathrm{dres}\left(\frac{\delta(u_1)}{u_1},\infty\right)+ m_2\cdot q\text{-}\mathrm{dres}\left(\frac{\delta(u_2)}{u_2},\infty\right),\end{equation} we have that $c\in\mathbb{Z}$ and every $q$-discrete residue of $m_1\frac{\delta(u_1)}{u_1}+m_2\frac{\delta(u_2)}{u_2}-c$ is $0$. By another application of Proposition~\ref{indsum}, there exists $f\in k$ satisfying \eqref{diagonal-urels} with $c$ as in \eqref{diagonal-crels}. \end{proof}

\begin{rem} \label{diagonalizable-rem} To compute the difference-differential Galois group $G$ for \eqref{difeq} when there exist at least two distinct solutions $u_1,u_2\in\bar{\mathbb{Q}}(x)$ to the Riccati equation \eqref{ric1}, we apply Proposition~\ref{galdiag} as follows. First, compute the $q$-discrete residues $r_i([\beta]_q):=q$-$\mathrm{dres}\left(\frac{\delta(u_i)}{u_i},[\beta]_q,1\right)$ at each $q^\mathbb{Z}$-orbit $[\beta]_q$ for $\beta\in\bar{\mathbb{Q}}^\times$ as in Definition~\ref{DEF:qres} (note these will be zero for any $\beta$ that is neither a zero nor a pole of $u_1$ or $u_2$). Then decide whether there exist relatively prime $m_1,m_2\in\mathbb{Z}$ such that $m_1r_1([\beta]_q)+m_2r_2([\beta]_q)=0$ for every $q^\mathbb{Z}$-orbit $[\beta]_q$ simultaneously (in general this will be an overdetermined linear system over $\bar{\mathbb{Q}}$, so the task is to decide whether there exists a non-zero solution in $\bar{\mathbb{Q}}^2$ and then whether such a solution can be taken to be in $\mathbb{Z}^2$). For any such pair $(0,0)\neq(m_1,m_2)\in\mathbb{Z}^2$, taking $c\in\mathbb{Z}$ as in \eqref{diagonal-crels} the proof of Proposition~\ref{galdiag} shows that there exists $f\in\bar{\mathbb{Q}}(x)$ satisfying \eqref{diagonal-urels}; it is not necessary to determine what the certificate $f$ actually is.

The $\mathbb{Z}$-module $M$ generated by all pairs $(m_1,m_2)\in\mathbb{Z}^2$ as in Proposition~\ref{galdiag}(ii) is free of rank $r\leq 2$. It follows from the proof of Corollary~\ref{diffprod} that $\mathbb{Z}^2/M$ is torsion-free, and therefore also free of rank $2-r$, since if $(dm_1,dm_2)\in M$ then $(m_1,m_2)\in M$ also for any $d\in\mathbb{Z}$. Thus, if the rank of $M$ is $r=2$ then $M=\mathbb{Z}^2$. It follows that the defining equations for $G$ arising from Proposition~\ref{galdiag}(ii) are given by either: a single pair $(m_1,m_2)$, unique up to multiplication by $\pm1$ and with the form of the defining equation determined by whether the corresponding value of $c$ in~\eqref{diagonal-crels} is $0$; or else the two relations corresponding to $(1,0)$ and $(0,1)$, with an additional relation occurring only in case $c_i:=q$-$\mathrm{dres}(\frac{\delta(u_i)}{u_i})\neq 0$ for both $i=1,2$, in which case we obtain an additional relation given by $\delta(\alpha_1^{d_1}\alpha_2^{\varepsilon d_2})=0$ with \begin{equation}\label{diagonal-remark-eq}d_i:=\frac{\mathrm{lcm}(|c_1|,|c_2|)}{c_i}\qquad \text{and}\qquad \varepsilon=\begin{cases} 1 \ \text{if} \ c_1c_2\in \mathbb{Z}_{<0};\\-1 \ \text{if} \ c_1c_2\in\mathbb{Z}_{>0}.\end{cases}\end{equation}

Having computed all possible relations arising from Proposition~\ref{galdiag}(ii), let us now show how to find the possible relations arising from Proposition~\ref{galdiag}(i), and thus determine all defining equations for $G\subseteq\mathbb{G}_m(C)^2$. We still denote by $M\subseteq \mathbb{Z}^2$ the $\mathbb{Z}$-submodule generated by pairs $(m_1,m_2)$ as in Proposition~\ref{galdiag}. If $M=\{(0,0)\}$ then $G=\mathbb{G}_m(C)^2$ as in Proposition~\ref{galdiag}(iii), so from now on we assume $M$ is not trivial. We saw above that either $M=\mathbb{Z}^2$ or else $M=\mathbb{Z}\cdot (m_1,m_2)$ with $\mathrm{gcd}(m_1,m_2)=1$.

Suppose $M=\mathbb{Z}\cdot(m_1,m_2)$. If the value of $c$ given in \eqref{diagonal-crels} is not $0$, then $G$ is defined by the single equation $\delta\left(\frac{\delta(\alpha_1^{m_1}\alpha_2^{m_2})}{\alpha_1^{m_1}\alpha_2^{m_2}}\right)=0$ as in Proposition~\ref{galdiag}(ii). On the other hand, if this $c=0$, then we must decide whether there exist: a primitive $n$-th root of unity $\zeta_n$, integers $r,s$ such that $0\leq r <s$ and $\mathrm{gcd}(r,s)=1$, and $g\in\bar{\mathbb{Q}}(x)^\times$ such that $u_1^{m_1}u_2^{m_2}=\zeta_nq_s^r\frac{\sigma(g)}{g}$. If so, then $G$ is defined by the single equation $(\alpha_1^{m_1}\alpha_2^{m_2})^\ell=1$ as in Proposition~\ref{galdiag}(i), where $\ell:=\mathrm{lcm}(n,s)$, the least common multiple of $n$ and $s$; otherwise, $\alpha_1^{m_1}\alpha_2^{m_2}$ has infinite order in $\mathbb{G}_m(C^\delta)$ for every $\left(\begin{smallmatrix}\alpha_1 & 0 \\ 0 & \alpha_2\end{smallmatrix}\right)\in G$, and $G$ is defined by the single equation $\delta(\alpha_1^{m_1}\alpha_2^{m_2})=0$ as in Proposition~\ref{galdiag}(ii) only.

If $M=\mathbb{Z}^2$, let again $c_i:=q$-$\mathrm{dres}(\frac{\delta(u_i)}{u_i})$. If exactly one $c_i$ is $0$, say $c_1=0\neq c_2$, then we must decide whether there exist:  a primitive $n$-th root of unity $\zeta_n$, integers $r,s$ such that $0\leq r <s$ and $\mathrm{gcd}(r,s)=1$, and $g\in\bar{\mathbb{Q}}(x)^\times$ such that $u_1=\zeta_nq_s^r\frac{\sigma(g)}{g}$. If so, then $G$ is defined by the equations: $\alpha_1^\ell=1$ as in Proposition~\ref{galdiag}(i), with $\ell:=\mathrm{lcm}(n,s)$, and $\delta\left(\frac{\delta(\alpha_2)}{\alpha_2}\right)=0$ as in Proposition~\ref{galdiag}(ii); otherwise, $G$ is defined instead by $\delta(\alpha_1)=0$ and $\delta\left(\frac{\delta(\alpha_2)}{\alpha_2}\right)=0$. The case where $c_2=0\neq c_1$ is analogous. If $c_1,c_2\neq 0$, then we must decide whether there exist:  a primitive $n$-th root of unity $\zeta_n$, integers $r,s$ such that $0\leq r <s$ and $\mathrm{gcd}(r,s)=1$, and $g\in\bar{\mathbb{Q}}(x)^\times$ such that $u_1^{d_1}u_2^{\varepsilon d_2}=\zeta_nq_s^r\frac{\sigma(g)}{g}$, with $d_1,d_2,\varepsilon$ defined as in \eqref{diagonal-remark-eq}. If so, then $G$ is defined by the equations $\delta\left(\frac{\delta(\alpha_i)}{\alpha_i}\right)=0$ for $i=1,2$ as in Proposition~\ref{galdiag}(ii), together with $(\alpha_1^{d_1}\alpha_2^{\varepsilon d_2})^\ell=1$ as in Proposition~\ref{galdiag}(i), where $\ell:=\mathrm{lcm}(n,s)$; otherwise, $G$ is defined by $\delta\left(\frac{\delta(\alpha_i)}{\alpha_i}\right)=0$ for $i=1,2$ only.

The case where $M=\mathbb{Z}^2$ and $c_i:=q$-$\mathrm{dres}(\frac{\delta(u_i)}{u_i})=0$ for both $i=1,2$ is similar in principle: we must decide whether there exist $m_1,m_2\in\mathbb{Z}$, a primitive $n$-th root of unity $\zeta_n$, integers $r,s$ such that $0\leq r<s$ and $\mathrm{gcd}(r,s)=1$, and $g\in\bar{\mathbb{Q}}(x)^\times$ such that $u_1^{m_1}u_2^{m_2}=\zeta_nq_s^r\frac{\sigma(g)}{g}$. If so, then $G$ is defined by $\delta(\alpha_i)=0$ for $i=1,2$ as in Proposition~\ref{galdiag}(ii), together with $(\alpha_1^{m_1}\alpha_2^{m_2})^\ell=1$ as in Proposition~\ref{galdiag}(i), where $\ell:=\mathrm{lcm}(n,s)$; otherwise, $G$ is defined by $\delta(\alpha_i)=0$ for $i=1,2$ only.

The problem of deciding whether $u:=u_1^{m_1}u_2^{m_2}=\zeta_nq_s^r\frac{\sigma(g)}{g}$ as above for a given \emph{single} pair $(m_1,m_2)$ is addressed in \cite[\S3]{hendriks:1997}: a straightforward modification of the algorithm given there allows us to compute a \emph{reduced form} $\tilde{u}=hx^n\frac{p}{q}$ where: $h\in\bar{\mathbb{Q}}^\times$; $n\in\mathbb{Z}$;  $p,q\in\bar{\mathbb{Q}}[x]$ are monic such that $\mathrm{gcd}(p,\sigma^m(q))=1$ for every $m\in\mathbb{Z}$; if $h=\zeta q^t$ for some root of unity $\zeta$ and $t\in\mathbb{Q}$, then $0\leq t<1$; and such that there exists $g\in\bar{\mathbb{Q}}(x)^\times$ with $u=\tilde{u}\frac{\sigma(g)}{g}$ for some $g\in\bar{\mathbb{Q}}(x)^\times$. Thus we only need to check whether $\tilde{u}=\zeta_nq_s^r$.

In the case where $M=\mathbb{Z}^2$ and $c_i:=q$-$\mathrm{dres}(\frac{\delta(u_i)}{u_i})=0$ for both $i=1,2$, one can show that the standard form $\tilde{u}_i=h_i$ with $h_i\in\bar{\mathbb{Q}}^\times$, and one needs to decide whether $h_1$ and $h_2$ are multiplicatively independent modulo $q^\mathbb{Z}$. We do not know how to produce \emph{a priori} bounds on the possible coefficients $(m_1,m_2)$ such that $h_1^{m_1}h_2^{m_2}\in q^\mathbb{Z}$ in general, so in this case only we offer no improvements on the algorithm in \cite[\S4.2]{hendriks:1997}. But in the remaining cases, we have reduced the computation of all the possible relations in Proposition~\ref{galdiag}(i) to checking a finite list of possibilities for $(m_1,m_2)\in\mathbb{Z}^2$, although this requires the ability to compute the $q$-discrete residues of $\frac{\delta(u_i)}{u_i}$ (cf.~\cite[\S2.2]{vanderput-singer:1997}).
\end{rem}


\section{Reducible non-diagonalizable groups} \label{reducible-sec}
We recall the notation introduced in the previous sections: $k=C(x)$, where $C$ is a $\delta$-closure of $\bar{\mathbb{Q}}$, $\sigma$ denotes the $C$-linear automorphism of $k$ defined by $\sigma(x)=qx$, and $\delta(x)=x$.

We now proceed to define the additional notation that we will use throughout this section. We will assume that there exists exactly one solution $u\in\bar{\mathbb{Q}}(x)$ to the Riccati equation \eqref{ric1}, so that the $\sigma$-Galois group $H$ for \eqref{difeq} is reducible but not diagonalizable as in Proposition~\ref{hendriks-possibilities}(3), and the difference operator implicit in \eqref{difeq} factors as $$\sigma^2+a\sigma+b=(\sigma-\tfrac{b}{u})\circ(\sigma - u),$$ as we saw in \S\ref{hendriks-sec}.
This means that there is a $C$-basis of solutions $\{y_1,y_2\}$ in any $\sigma\delta$-PV ring $R$ for \eqref{difeq} such that $y_1,y_2\neq 0$ satisfy $\sigma(y_1)=uy_1$ and $\sigma(y_2)-uy_2=y_0$, where $y_0\neq 0$ satisfies $\sigma(y_0)=\tfrac{b}{u}y_0$. A fundamental solution matrix for \eqref{mateq} is given by \begin{equation}\label{fundsol}\begin{pmatrix} y_1 & y_2 \\ \sigma(y_1) & \sigma(y_2)\end{pmatrix}=\begin{pmatrix} y_1 & y_2 \\ uy_1 & uy_2+y_0\end{pmatrix}.\end{equation}

If we now let $A=\left(\begin{smallmatrix} 0 & 1 \\ -b & -a\end{smallmatrix}\right)$, $T=\left(\begin{smallmatrix} 1-u & 1 \\ -u & 1\end{smallmatrix}\right)$, and $v=\tfrac{b}{u}=-\sigma(u)-a$ (since $u$ satisfies \eqref{ric1}), we have that \[\sigma(T)AT^{-1}=\begin{pmatrix} 1-\sigma(u) & 1 \\ -\sigma(u) & 1\end{pmatrix}\begin{pmatrix} 0 & 1 \\ -b & -a \end{pmatrix}\begin{pmatrix}1 & -1 \\ u & 1-u\end{pmatrix}  =\begin{pmatrix}u & 1-u+v \\ 0 & v\end{pmatrix}=:B.\] Therefore, the systems \eqref{mateq} and $\sigma(Z)=BZ$ are equivalent (in the sense of Definition~\ref{equivalent-def}), and a fundamental solution matrix for the latter system is given by $$TY=\begin{pmatrix} 1-u & 1 \\ -u & 1\end{pmatrix}\begin{pmatrix} y_1 & y_2 \\ uy_1 & uy_2+y_0\end{pmatrix}=\begin{pmatrix} y_1 & y_2+y_0 \\ 0 & y_0\end{pmatrix}=Z.$$

For any $\gamma\in H$, the $\sigma$-Galois group for \eqref{difeq}, we have that \begin{equation}\label{h-emb} \gamma\begin{pmatrix} y_1 & y_2+y_0 \\ 0 & y_0\end{pmatrix}=\begin{pmatrix} y_1 & y_2+y_0 \\ 0 & y_0\end{pmatrix}\begin{pmatrix} \alpha_\gamma & \xi_\gamma \\ 0 & \lambda_\gamma\end{pmatrix} = \begin{pmatrix} \alpha_\gamma y_1 & \xi_\gamma y_1+\lambda_\gamma y_2+\lambda_\gamma y_0 \\ 0 & \lambda_\gamma y_0\end{pmatrix},\end{equation} and therefore the action of $H$ on the solutions is defined by \begin{equation}\label{h-action}\gamma(y_1)=\alpha_\gamma y_1; \qquad \gamma(y_0)=\lambda_\gamma y_0; \qquad \text{and}\qquad \gamma(y_2)=\lambda_\gamma y_2 + \xi_\gamma y_1.\end{equation} It will be convenient to define the auxiliary elements \begin{equation} \label{zw-def} w=\frac{y_0}{uy_1}  \qquad \text{and}\qquad z=\frac{y_2}{y_1},\end{equation} on which $\sigma$ acts via \begin{gather} \label{zw-sigma}  \sigma(w)=\frac{b}{u\sigma(u)}w; \qquad \sigma(z)=z+w, \intertext{and $H$ acts via}  \label{zw-h}   \gamma(w)=\frac{\lambda_\gamma}{\alpha_\gamma}w; \qquad   \gamma(z)=\frac{\lambda_\gamma}{\alpha_\gamma}z+\frac{\xi_\gamma}{\alpha_\gamma}.\end{gather}
We observe that the $\sigma$-PV ring $$S=k[y_1, y_2+y_0, y_0, (y_1y_0)^{-1}]=k[y_1,w, z, (y_1w)^{-1}]$$ and the $\sigma\delta$-PV ring $$R=k\{y_1,y_2+y_0,y_0,(y_1y_0)^{-1}\}_\delta=k\{y_1,w, z, (y_1w)^{-1}\}_\delta.$$

Our computation of the $\sigma\delta$-Galois group $G$ for \eqref{difeq} in this section will be accomplished by studying the action of $G$ on $y_1$, $w$, and $z$. We begin by defining the \emph{unipotent radicals} \begin{equation} \label{uniprad-def} R_u(H)=H\cap\left\{\begin{pmatrix} 1& \xi \\ 0 & 1\end{pmatrix} \ \middle| \ \xi\in C\right\}\quad\text{and}\quad R_u(G)=G\cap\left\{\begin{pmatrix} 1& \xi \\ 0 & 1\end{pmatrix} \ \middle| \ \xi\in C\right\},\end{equation} and observe that $R_u(H)$ (resp., $R_u(G)$) is an algebraic (resp., differential algebraic) subgroup of $\mathbb{G}_a(C)$, the additive group of $C$. By \cite[Thm.~13(2)]{hendriks:1997}, $R_u(H)=\mathbb{G}_a(C)$ if and only if there exists exactly one solution $u\in k$ to \eqref{ric1}. We observe that $$R_u(G)=\{\gamma\in G \ | \ \gamma(y_i)=y_i \ \text{for} \ i=0,1\}.$$The reductive quotient $$G/R_u(G)\simeq\left\{\begin{pmatrix} \alpha_\gamma & 0 \\ 0 & \lambda_\gamma\end{pmatrix} \ \middle| \ \gamma\in G\right\}$$ is the $\sigma\delta$-Galois group corresponding to the matrix equation \begin{equation}\label{redeq}\sigma(Y)=\begin{pmatrix} u & 0 \\ 0 & v \end{pmatrix}Y,\end{equation} which we compute with Proposition~\ref{galdiag} and Remark~\ref{diagonalizable-rem}.

In the following result, we compute the defining equations for the $\sigma\delta$-Galois group $G$ for \eqref{difeq} in a special case. Recall that $u\in\bar{\mathbb{Q}}(x)$ denotes the unique solution to the Riccati equation \eqref{ric1}, $H$ denotes the $\sigma$-Galois group for \eqref{difeq}, and $w$ is as in \eqref{zw-def}.

\begin{prop} \label{wrat} Suppose there is exactly one solution $u\in k$ to \eqref{ric1} and $H$ is commutative. Then $H$ is a subgroup of \begin{equation} \label{bigut}\mathbb{G}_m(C)\times\mathbb{G}_a(C)=\left\{\begin{pmatrix} \alpha & \xi \\ 0 & \alpha\end{pmatrix} \ \middle| \ \alpha,\xi\in C, \ \alpha\neq 0\right\}\end{equation} with $R_u(H)=\mathbb{G}_a(C)$. Moreover, there exists $ w\in\bar{\mathbb{Q}}(x)$ satisfying \eqref{zw-sigma}, and $G$ is the subgroup of \eqref{bigut} defined by the following conditions on $\alpha$ and $\xi$.
\begin{enumerate}
\item[(i)] There exist $m\in\mathbb{N}$ and $f\in\bar{\mathbb{Q}}(x)^\times$ such that $u^m=\frac{\sigma(f)}{f}$ if and only if $\alpha^m=1$.
\item[(ii)] There exist $c\in\mathbb{Z}$ and $f\in \bar{\mathbb{Q}}(x)$ such that $\frac{\delta(u)}{u}=\sigma(f)-f+c$ if and only if $\delta\left(\frac{\delta(\alpha)}{\alpha}\right)=0$. Moreover, $c=0$ if and only if $\delta(\alpha)=0$.
\item[(iii)] There exist: $c\in\bar{\mathbb{Q}}$; $f\in \bar{\mathbb{Q}}(x)$; and a linear $\delta$-polynomial $\mathcal{L}\in \bar{\mathbb{Q}}\{Y\}_\delta$ such that $\mathcal{L}(\frac{\delta(u)}{u})-w=\sigma(f)-f+c$ if and only if $\delta\left(\frac{\xi}{\alpha}\right)=\mathcal{L}\left(\delta\left(\frac{\delta(\alpha)}{\alpha}\right)\right)$. Moreover, $c=0$ if and only if $\xi=\alpha\mathcal{L}(\frac{\delta(\alpha)}{\alpha})$.
\item[(iv)] If none of the conditions above is satisfied, then $G=H=\mathbb{G}_m(C)\times\mathbb{G}_a(C)$.
\end{enumerate}
\end{prop}

\begin{proof} First recall that when there is exactly one solution $u\in k$ to \eqref{ric1} the $\sigma$-Galois group $H$ of \eqref{difeq} is reducible but not diagonalizable by \cite[Thm.~13]{hendriks:1997}, and therefore $H$ is a non-diagonalizable subgroup of \[\mathbb{G}_m(C)\ltimes\mathbb{G}_a(C)=\left\{\begin{pmatrix} \alpha & \xi \\ 0 & \lambda\end{pmatrix} \ \middle| \ \alpha,\xi,\lambda\in C, \ \alpha\lambda\neq 0\right\}.\] In particular, $R_u(H)=\mathbb{G}_a(C)$ and a straightforward computation shows that $H$ is commutative if and only if it is actually a subgroup of \eqref{bigut}. We recall the notation introduced at the beginning of this section: $v=\frac{b}{u}$, $\{y_1,y_2\}$ is a $C$-basis of solutions for \eqref{difeq} such that $\sigma(y_1)=uy_1$ and $\sigma(y_2)-uy_2=y_0$, where $y_0\neq 0$ satisfies $\sigma(y_0)=vy_0$. The embedding $H\hookrightarrow \mathrm{GL}_2(C):\gamma\mapsto M_\gamma$ is as in \eqref{h-emb}, and the action of $H$ on the solutions is given in \eqref{h-action}. The auxiliary elements $w$ and $z$ are defined as in \eqref{zw-def}; they are acted upon by $\sigma$ as in \eqref{zw-sigma} and by $H$ as in \eqref{zw-h}. The relation $\gamma(w)=\frac{\lambda_\gamma}{\alpha_\gamma}w$ for each $\gamma\in H$ from \eqref{zw-h}, together with Theorem~\ref{correspondence}, imply that $w\in k$. Since $\sigma(w)=\frac{b}{u\sigma(u)}w$ from \eqref{zw-sigma} and $b,u\in\bar{\mathbb{Q}}(x)$, if $w\in k$ we may actually take $w\in \bar{\mathbb{Q}}(x)$ by \cite[Lem.~2.5]{hardouin:2008} (cf.~Remark~\ref{c-ok}). Thus, if we can find $f\in k$ witnessing the relations in items (i) or (ii), then we may take $f\in\bar{\mathbb{Q}}(x)$, as already discussed in the proof of Proposition~\ref{galdiag}, and similarly if we can find $f\in k$ and $c\in C$ witnessing the relation in item~(iii), then we may take $f\in\bar{\mathbb{Q}}(x)$ and $c\in\bar{\mathbb{Q}}$.

Items (i) and (ii) were already established in Proposition~\ref{galdiag}. Let us prove item~(iii). Setting $\frac{\delta(y_1)}{y_1}=:g\in R$ we have that \begin{equation}\label{reducible-urels} \sigma(g)-g=\frac{\delta(u)}{u}\qquad\text{and}\qquad \gamma(g)=g+\frac{\delta(\alpha_\gamma)}{\alpha_\gamma}\end{equation} for $\gamma\in H$ (cf.~the proof of Proposition~\ref{galdiag}). On the other hand, the actions of $\sigma$ and $\gamma\in H$ on the element $z\in R$ defined in \eqref{zw-def} in this case is given by \begin{equation}\label{reducible-wrels} \sigma(z)-z=w\qquad \text{and} \qquad \gamma(z)=z+\frac{\xi_\gamma}{\alpha_\gamma}.\end{equation} Consider the relation stipulated in item~(iii): \begin{equation}\label{reducible-uwrel}\mathcal{L}\left(\frac{\delta(u)}{u}\right)-w=\sigma(f)-f+c,\end{equation} where $\mathcal{L}\in\bar{\mathbb{Q}}\{Y\}_\delta$ is a linear differential polynomial, $f\in\bar{\mathbb{Q}}(x)$, and $c\in\bar{\mathbb{Q}}$. It follows from \eqref{reducible-urels} and \eqref{reducible-wrels} that, for any linear differential polynomial $\mathcal{L}\in C\{Y\}_\delta$ and $\gamma\in G$, we have that \begin{gather*} \sigma\left(\mathcal{L}(g) -z\right)-\left(\mathcal{L}(g)-z\right)=\mathcal{L}\left(\frac{\delta(u)}{u}\right)-w;\qquad\text{and}\\ \gamma\left(\mathcal{L}(g)-z\right)=\left(\mathcal{L}(g)-z\right) + \mathcal{L}\left(\frac{\delta(\alpha_\gamma)}{\alpha_\gamma}\right)-\frac{\xi_\gamma}{\alpha_\gamma}.\end{gather*}

Suppose there exist $f\in\bar{\mathbb{Q}}(x)$ and a linear differential polynomial $\mathcal{L}\in\bar{\mathbb{Q}}\{Y\}_\delta$ satisfying \eqref{reducible-uwrel} with $c=0$. Then $\mathcal{L}(g)-z-f\in k^\sigma$, which implies that $\mathcal{L}(g)-z\in k$, and therefore $\mathcal{L}\left(\frac{\delta(\alpha_\gamma)}{\alpha_\gamma}\right)=\frac{\xi_\gamma}{\alpha_\gamma}$ for every $\gamma\in G$ by Theorem~\ref{correspondence}. On the other hand, if $\mathcal{L}\in\bar{\mathbb{Q}}\{Y\}_\delta$ is a linear differential polynomial such that $\mathcal{L}\left(\frac{\delta(\alpha_\gamma)}{\alpha_\gamma}\right)=\frac{\xi_\gamma}{\alpha_\gamma}$ for every $\gamma\in G$, then $\mathcal{L}(g)-z=f\in k$ satisfies \eqref{reducible-uwrel} with $c=0$.

More generally, suppose there exist $f\in \bar{\mathbb{Q}}(x)$, $c\in \bar{\mathbb{Q}}$, and a linear differential polynomial $\mathcal{L}\in\bar{\mathbb{Q}}\{Y\}_\delta$ satisfying \eqref{reducible-uwrel}. Then we see that $\mathcal{L}(\delta(g))-\delta(z)-\delta(f)\in k_\sigma$, and therefore $\mathcal{L}(\delta(g))-\delta(z)\in k$, which implies that \begin{equation}\label{reducible-lrel}\mathcal{L}\left(\delta\left(\frac{\delta(\alpha_\gamma)}{\alpha_\gamma}\right)\right)=\delta\left(\frac{\xi_\gamma}{\alpha_\gamma}\right)\end{equation} for every $\gamma\in G$. On the other hand, if $\mathcal{L}\in\bar{\mathbb{Q}}\{Y\}_\delta$ is a linear differential polynomial such that \eqref{reducible-lrel} holds for every $\gamma\in G$, then $\mathcal{L}(\delta(g))-\delta(g)=:h\in k$, and therefore the element $\mathcal{L}(g)-z\in R$ is differentially dependent over $k$. It then follows from \cite[Prop.~3.10(2.a)]{hardouin-singer:2008} that there exist $f\in k$ and $c\in C^\delta=\bar{\mathbb{Q}}$ satisfying \eqref{reducible-uwrel}. This concludes the proof of item~(iii).

By \cite[Cor.~3.2]{hardouin-singer:2008}, $g$ and $z$ are differentially dependent over $k$ if and only if there exist linear differential polynomials $\mathcal{L}_1,\mathcal{L}_2\in\bar{\mathbb{Q}}\{Y\}_\delta$, not both zero, and $\tilde{f}\in\bar{\mathbb{Q}}(x)$, such that \begin{equation}\label{reducible-2ops}\mathcal{L}_1\left(\frac{\delta(u)}{u}\right)-\mathcal{L}_2(w)=\sigma(\tilde{f})-\tilde{f}.\end{equation} Hence if there do not exist such $\mathcal{L}_i$ and $\tilde{f}$, the elements $g,z\in R$ are differentially independent over $k$, which implies that $G=\mathbb{G}_m(C)\ltimes\mathbb{G}_a(C)$ by \cite[Prop.~6.26]{hardouin-singer:2008}. Thus, assume there do exist $\mathcal{L}_1,\mathcal{L}_2\in\bar{\mathbb{Q}}\{Y\}_\delta$, not both zero, and $\tilde{f}$ satisfying \eqref{reducible-2ops}. If $\mathcal{L}_2=0$, then $\mathcal{L}_1\neq 0$ and it follows from \eqref{reducible-2ops} and \eqref{reducible-urels} that $g$ is differentially dependent over $k$, and therefore so is $y_1$. By Corollary~\ref{diffprod}, this implies that there exist $f\in\bar{\mathbb{Q}}(x)$ and $c\in\mathbb{Z}$ such that $\frac{\delta(u)}{u}=\sigma(f)-f+c$, as in item~(ii). To prove item~(iv), let us show that if there exist linear differential polynomials $\mathcal{L}_1,\mathcal{L}_2\in\bar{\mathbb{Q}}\{Y\}_\delta$ with $\mathcal{L}_2\neq 0$ and $\tilde{f}\in\bar{\mathbb{Q}}(x)$ satisfying \eqref{reducible-2ops}, then we can construct a linear $\delta$-polynomial $\mathcal{L}\in \bar{\mathbb{Q}}\{Y\}_\delta$ and $c \in \bar{\mathbb{Q}}$ such that 
$$\mathcal{L}\left(\frac{\delta(u)}{u}\right)-w=\sigma(f)-f + c\quad\text{for some}\quad f\in \bar{\mathbb{Q}}(x),$$  
as in item~(iii). Let $\mathrm{ord}(\mathcal{L}_i)=m_i$ and $\mathcal{L}_i=\sum_{j=0}^{m_i}c_{i,j}\delta^j(Y)$ for $i=1,2$; if $\mathcal{L}_1=0$, we set $m_1=0$, and we adopt the convention that $c_{i,j}:=0$ for every $j>m_i$. 
By Proposition~\ref{indsum}, the existence of $\tilde{f}\in k$ as in \eqref{reducible-2ops} implies that 
\begin{equation}\label{wrat-5} 
0=q\text{-}\mathrm{dres} \left(\mathcal{L}_1\left(\frac{\delta(u)}{u}\right)-\mathcal{L}_2(w), [\beta]_q, j\right)
\end{equation} 
for every $q^\mathbb{Z}$-orbit $[\beta]_q$ with $\beta\in\bar{\mathbb{Q}}^\times$ and every $j\in\mathbb{N}$. 
Let $r\in\mathbb{N}$ be the largest order such that 
$$q\text{-}\mathrm{dres} (w,[\beta]_q,r)\neq 0 \quad \text{for some} \quad q^\mathbb{Z}\text{-orbit} \quad [\beta]_q.$$ 
Then it follows from~\eqref{wrat-5} and Lemma~\ref{qres-computation} that, for each $q^\mathbb{Z}$-orbit $[\beta]_q$ with $\beta\in\bar{\mathbb{Q}}^\times$, 
the $q$-discrete residues 
\begin{gather*}
c_{1,m_2+r-1}(-1)^{m_2+r-1}(m_2+r-1)! \beta^{m_2 + r - 1} q\text{-}\mathrm{dres} \!\bigl(\tfrac{\delta(u)}{u},[\beta]_q,1\bigr) = q\text{-}\mathrm{dres} \!\Bigl(\mathcal{L}_1\bigl(\tfrac{\delta(u)}{u}\bigr),[\beta]_q,m_2+r\Bigr) 
\intertext{and} 
c_{2,m_2}(-1)^{m_2}\frac{(m_2+r)!}{(r-1)!} \beta^{r - 1} q\text{-}\mathrm{dres} (w,[\beta]_q, r) = q\text{-}\mathrm{dres} (\mathcal{L}_2(w),[\beta]_q, m_2+r)
\end{gather*} 
are equal. Since $\beta \neq 0$, the above equality is equivalent to 
\begin{equation} \label{wrat-6}
\frac{c_{1,m_2+r-1}}{c_{2,m_2}} (-1)^{r - 1} (r - 1)! \beta^{r - 1} q\text{-}\mathrm{dres} \!\bigl(\tfrac{\delta(u)}{u},[\beta]_q,1\bigr) = q\text{-}\mathrm{dres} (w,[\beta]_q, r)
\end{equation}
Set $c_{r - 1} = \frac{c_{1,m_2+r-1}}{c_{2,m_2}}$. Then~\eqref{wrat-6} is equivalent to 
\[
q\text{-}\mathrm{dres} \left(c_{r-1}\delta^{r-1}\left(\frac{\delta(u)}{u}\right)-w, [\beta]_q, r\right) = 0
\] for every $q^\mathbb{Z}$-orbit $[\beta]_q$ with $\beta\in\bar{\mathbb{Q}}^\times$ simultaneously.

We continue by taking the next highest $r'\leq r-1$ such that $q\text{-}\mathrm{dres} (w,[\beta]_q,r')\neq 0$ for some $[\beta]_q$, and proceed as above to find the coefficient $c_{r'-1}\in \bar{\mathbb{Q}}$ of $\mathcal{L}$ such that 
\[
q\text{-}\mathrm{dres} \left(c_{r-1}\delta^{r-1}\left(\frac{\delta(u)}{u}\right) + c_{r'-1} \delta^{r'- 1} \left(\frac{\delta(u)}{u}\right) - w, [\beta]_q, r'\right) = 0.
\]
Eventually we will have constructed a linear $\delta$-polynomial $\mathcal{L}\in \bar{\mathbb{Q}}\{Y\}_\delta$ such that $$q\text{-}\mathrm{dres} \left(\mathcal{L}\left(\frac{\delta(u)}{u}\right)-w,[\beta]_q, j\right)=0$$ for every $q^\mathbb{Z}$-orbit $[\beta]_q$ with $\beta\in\bar{\mathbb{Q}}^\times$ and every $j\in\mathbb{N}$. 
Set $c = q\text{-}\mathrm{dres} \left(\mathcal{L}\left(\frac{\delta(u)}{u}\right)-w, \infty\right)\in\bar{\mathbb{Q}}$. Then it follows from Proposition~\ref{indsum} that $\mathcal{L}(\frac{\delta(u)}{u})-w-c=\sigma(f)-f$ for some $f\in k$. By \cite[Lem.~2.4]{hardouin:2008} (cf.~Remark~\ref{c-ok}), we may take $f\in \bar{\mathbb{Q}}(x)$, so we are indeed in case (iii), as we wanted to show.
\end{proof}

The main ideas for the proof of the following result were communicated to the first author in \cite{singer-letter}, during the development of the algorithm in \cite{arreche:2017}.

\begin{prop}\label{wnotrat}Suppose there exists exactly one solution $u\in\bar{\mathbb{Q}}(x)$ to \eqref{ric1} and $H$ is not commutative. Then $R_u(G)=R_u(H)=\mathbb{G}_a(C)$.
\end{prop}

\begin{proof}  We recall the notation introduced at the beginning of this section: $u\in\bar{\mathbb{Q}}(x)$ is the unique solution in $k$ to the Riccati equation~\eqref{ric1}, $v=\frac{b}{u}$, $\{y_1,y_2\}$ is a $C$-basis of solutions for \eqref{difeq} such that $\sigma(y_1)=uy_1$ and $\sigma(y_2)-uy_2=y_0$, where $y_0\neq 0$ satisfies $\sigma(y_0)=vy_0$. The embedding $G\hookrightarrow \mathrm{GL}_2(C):\gamma\mapsto M_\gamma$ is as in \eqref{h-emb}, and the action of $G$ on the solutions is given in \eqref{h-action}. The auxiliary elements $w$ and $z$ are defined as in \eqref{zw-def}; they are acted upon by $\sigma$ as in \eqref{zw-sigma} and by $G$ as in \eqref{zw-h}.

By Proposition~\ref{classification}, either $R_u(G)=\mathbb{G}_a(C)$, or else \begin{equation}\label{radop}R_u(G)=\left\{\begin{pmatrix} 1 & \xi \\ 0 & 1\end{pmatrix} \ \middle| \ \beta\in C, \ \mathcal{L}(\xi)=0\right\}.\end{equation} for some nonzero, monic linear $\delta$-polynomial $\mathcal{L}\in C\{Y\}_\delta$. Since $R_u(G)$ is normal in $G$, this implies that (cf. \cite[Lem.~3.6]{hardouin-singer:2008})$$M_\gamma\begin{pmatrix} 1 & \beta \\ 0 &1\end{pmatrix}M_\gamma^{-1}=\begin{pmatrix}\alpha_\gamma & \xi_\gamma \\ 0 & \lambda_\gamma\end{pmatrix}\begin{pmatrix} 1 & \xi \\ 0 &1\end{pmatrix}\begin{pmatrix}\alpha_\gamma^{-1} & -\alpha_\gamma^{-1}\lambda_\gamma^{-1}\xi_\gamma \\ 0 & \lambda_\gamma^{-1}\end{pmatrix}=\begin{pmatrix} 1 & \alpha_\gamma\lambda_\gamma^{-1}\xi \\ 0 & 1\end{pmatrix}\in R_u(G)$$ for each $\gamma\in G$ and $\left(\begin{smallmatrix} 1 & \xi \\ 0 & 1\end{smallmatrix}\right)\in R_u(G)$. If $\mathcal{L}$ is as in \eqref{radop}, then $\mathcal{L}(\xi)=0\Rightarrow\mathcal{L}(\alpha_\gamma\lambda_\gamma^{-1}\xi)=0$. By \cite[Lem.~3.7]{hardouin-singer:2008}, this implies that if $\mathrm{ord}(\mathcal{L})\neq 0$, then $\delta(\alpha_\gamma\lambda_\gamma^{-1})=0$ for every $\gamma\in G$. But since $\mathcal{L}\neq 0$, $\mathrm{ord}(\mathcal{L})=0$ if and only if $R_u(G)=\{0\}$, which is impossible, for then we would have that \[G\simeq G/R_u(G)\simeq\left\{\begin{pmatrix}\alpha_\gamma & 0 \\ 0 & \lambda_\gamma\end{pmatrix} \ \middle| \ \gamma\in G\right\}\] is commutative, and since $G$ is Zariski-dense in $H$ by Proposition~\ref{dense}, this would force $H$ to be commutative also, contradicting our hypotheses.

We proceed by contradiction: assuming $R_u(G)\neq\mathbb{G}_a(C)$, we will show that $R_u(H)=\{0\}$, contradicting our hypotheses. We have shown above that if $R_u(G)\neq\{0\}$ then there exists a monic linear $\delta$-polynomial $\mathcal{L}\in C\{Y\}_\delta$ with $\mathrm{ord}(\mathcal{L})\geq 1$ such that $R_u(G)$ is as in \eqref{radop} and $\delta\left(\frac{\alpha_\gamma}{\lambda_\gamma}\right)=0$ for every $\gamma\in G$. It follows from \eqref{zw-h} and Theorem~\ref{correspondence} that the group $\{\lambda_\gamma\alpha_\gamma^{-1} \ | \ \gamma\in G\}\subseteq\mathbb{G}_m(\bar{\mathbb{Q}})$ is the $\sigma\delta$-Galois group for the system \[\sigma(W)=\left(\frac{b}{u\sigma(u)}\right)W,\] which by Proposition~\ref{dconst} must be integrable over $k$ in the sense of \cite[Def.~3.3]{arreche-singer:2016}. It is shown in \cite[Prop.~3.6]{arreche-singer:2016} that this system must then be integrable over $\bar{\mathbb{Q}}(x)$, and therefore by \cite[Thm.~2]{schaefke-singer:2019}, there exist $t\in\bar{\mathbb{Q}}(x)^\times$ and $c,d\in\bar{\mathbb{Q}}$ with $c\neq 0$ such that $\tilde{w}:=t^{-1}w$ satisfies \begin{align*}\sigma(\tilde{w}) &=c\tilde{w};\\ \delta(\tilde{w})&=d\tilde{w}.\end{align*} It is convenient to point out now that $c\neq q^r$ for any $r\in \mathbb{Z}$, because otherwise we would have $\tilde{w}=ex^r$ for some $e\in C$, which would imply that $w\in k$, contradicting our hypothesis that $H$ is not commutative (cf.~the proof of Proposition~\ref{wrat}: $H$ is commutative if and only if $\alpha_\gamma=\lambda_\gamma$ for every $\gamma\in H$ if and only if $w\in k$ by \eqref{zw-h} and Theorem~\ref{correspondence}). We will need to use the fact that $c\notin q^\mathbb{Z}$ at the end of the proof.

We claim that \begin{equation}\label{w-operator}\tilde{w}^{-1}\mathcal{L}(w)=:f_\mathcal{L}\in k,\qquad\text{and moreover}\qquad f_\mathcal{L}=\tilde{\mathcal{L}}(t)\end{equation} for some linear differential polynomial $0\neq\tilde{\mathcal{L}}\in C\{Y\}_\delta$. In fact, this is true for any non-zero linear differential polynomial in $C\{Y\}_\delta$, not just for the specific $\mathcal{L}\in C\{Y\}_\delta$ in \eqref{radop}. It suffices to show that $\tilde{w}^{-1}\delta^n(w)$ belongs to the $C$-linear span $\mathcal{D}$ of $\{\delta^j(t) \ | \ j\in\mathbb{Z}_{\geq 0}\}$ for every $n\in\mathbb{N}$. We prove this by induction: the case $n=0$ is clear, since $\tilde{w}^{-1}w=t\in \mathcal{D}$. Assuming that $\tilde{w}^{-1}\delta^n(w)=f_n\in\mathcal{D}$, we see that \[\tilde{w}^{-1}\delta^{n+1}(w)=\tilde{w}^{-1}\delta(\delta^n(w))=\tilde{w}^{-1}\delta(f_n\tilde{w})=\delta(f_n)+df_n\in \mathcal{D}\] as well. Moreover, this computation also shows that $\mathcal{L},\tilde{\mathcal{L}}\in C\{Y\}_\delta$ have the same order and the same leading coefficient.

By \eqref{zw-sigma} and \eqref{zw-h}, the element $\mathcal{L}(z)\in R$ satisfies \[\sigma(\mathcal{L}(z))-\mathcal{L}(z)=\mathcal{L}(w),\qquad \text{and}\qquad \gamma(\mathcal{L}(z))=\frac{\lambda_\gamma}{\alpha_\gamma}\mathcal{L}(z)+\mathcal{L}\left(\frac{\xi_\gamma}{\alpha_\gamma}\right)\] for every $\gamma\in G$, since $\delta(\lambda_\gamma\alpha_\gamma^{-1})=0$ for $\gamma\in G$. Hence $\gamma(\mathcal{L}(z))=\mathcal{L}(z)$ for every $\gamma\in R_u(G)$, and therefore by Theorem~\ref{correspondence} we have that $\mathcal{L}(z)\in k\langle y_0,y_1\rangle_\delta=:F$, the total ring of fractions of the $\sigma\delta$-PV ring $k\{y_0,y_1,(y_0y_1)^{-1}\}_\delta$ for \eqref{redeq}; we emphasize that the latter ring is not necessarily a domain, so $F$ is not necessarily a field.

For $\gamma\in G/R_u(G)\simeq\mathrm{Gal}_{\sigma\delta}(F/k)=:\bar{G}$ given by $\left(\begin{smallmatrix}\alpha_\gamma & 0 \\ 0 & \lambda_\gamma\end{smallmatrix}\right)\in\mathbb{G}_m(C)^2$, let \begin{equation}\label{cocycle-def} \tau_\gamma:=\gamma(\tilde{w}^{-1}\mathcal{L}(z))-\tilde{w}^{-1}\mathcal{L}(z),\end{equation} where we note that since $\mathcal{L}(z)\in F$ is fixed by $R_u(G)$, the action of the reductive quotient $\bar{G}$ on $\mathcal{L}(z)$ is well-defined. We claim that $\{\tau_\gamma \ | \ \gamma\in\bar{G}\}$ is a $1$-cocycle of $\bar{G}$ with values in the $\bar{G}$-module $M:=C\cdot \tilde{w}^{-1}$ (see~\cite[VI.10]{lang:2002}). Since $M$ is the solution space for $\sigma(W)=c^{-1}W$ in $F$, it is clear that $M$ is stabilized by $\bar{G}$. Moreover, it follows from \eqref{w-operator} that \[\sigma(\tau_\gamma)=((\gamma-1)\circ\sigma)(\tilde{w}^{-1}\mathcal{L}(z))=(\gamma-1)\bigl(c^{-1}\tilde{w}^{-1}\mathcal{L}(z)+c^{-1}\tilde{w}^{-1}\mathcal{L}(w)\bigr)=c^{-1}\tau_\gamma,\] since $c^{-1}\tilde{w}^{-1}\mathcal{L}(w)=c^{-1}f_\mathcal{L}\in k$ and therefore $\gamma(c^{-1}f_\mathcal{L})=c^{-1}f_\mathcal{L}$ for every $\gamma\in\bar{G}$. Hence $\tau_\gamma\in M$ for each $\gamma\in\bar{G}$. To verify the cocycle condition, note that for $\gamma,\theta\in\bar{G}$ we have that \begin{align*}\tau_{\gamma\theta}=\gamma\theta(\tilde{w}^{-1}\mathcal{L}(z))-\tilde{w}^{-1}\mathcal{L}(z) &=\gamma\bigl(\theta(\tilde{w}^{-1}\mathcal{L}(z))-\tilde{w}^{-1}\mathcal{L}(z)\bigr)+\bigl(\gamma(\tilde{w}^{-1}\mathcal{L}(z))-\tilde{w}^{-1}\mathcal{L}(z)\bigr)\\ &=\gamma(\tau_\theta)+\tau_\gamma.\end{align*}

Since $G$ is not commutative (for otherwise $H$ would be commutative, as discussed above and contrary to our hypotheses), there exists $\gamma\in\bar{G}$ such that $\alpha_\gamma\neq\lambda_\gamma$, and therefore $m\mapsto \gamma(m)-m$ is a $\bar{G}$-automorphism of $M$ for such a $\gamma\in\bar{G}$, since $\bar{G}$ is commutative. By Sah's Lemma~\cite[Lem.~VI.10.2]{lang:2002}, the cohomology group $H^1(\bar{G},M)=\{0\}$, and in particular $\{\tau_\gamma \ | \ \gamma\in\bar{G}\}$ is a $1$-coboundary, i.e., there exists $e\tilde{w}^{-1}\in C\cdot\tilde{w}^{-1}=M$ such that $\tau_\gamma=\gamma(e\tilde{w}^{-1})-e\tilde{w}^{-1}$. It follows from the definition of $\tau_\gamma$ in \eqref{cocycle-def} that \[\gamma(\tilde{w}^{-1}\mathcal{L}(z)-e\tilde{w}^{-1})=\tilde{w}^{-1}\mathcal{L}(z)-e\tilde{w}^{-1}\] for every $\gamma\in\bar{G}$, which implies that $\tilde{w}^{-1}\mathcal{L}(z)-e\tilde{w}^{-1}=:g\in k$ by Theorem~\ref{correspondence}. Hence\[f_\mathcal{L}\tilde{w}=\mathcal{L}(w)=\sigma(\mathcal{L}(z))-\mathcal{L}(z)=\sigma(g\tilde{w})-g\tilde{w}=(c\sigma(g)-g)\tilde{w},\] and therefore, since $c\in\bar{\mathbb{Q}}^\times$, \[\tilde{\mathcal{L}}(c^{-1}t)=c^{-1}f_\mathcal{L}=\sigma(g)-c^{-1}g,\] where $0\neq\tilde{\mathcal{L}}\in C\{Y\}_\delta$ is the linear differential polynomial defined implicitly in \eqref{w-operator}. Since $c\notin q^{\mathbb{Z}}$, it follows from \cite[Prop.~6.4(2)]{hardouin-singer:2008} that there exists $h\in k$ such that \[c^{-1}t=\sigma(h)-c^{-1}h.\] But then $h\tilde{w}$ satisfies \[\sigma(h\tilde{w})-h\tilde{w}=(c\sigma(h)-h)\tilde{w}=t\tilde{w}=w,\] and therefore $\sigma(z-h\tilde{w})-(z-h\tilde{w})=0$ by \eqref{zw-sigma}, which implies that $z-h\tilde{w}\in C$ and therefore $z\in k[w]$ is fixed by $R_u(H)$. But $\gamma(z)=z+\xi_\gamma$ for every $\gamma\in R_u(H)$, and therefore $R_u(H)=\{0\}$, which contradicts our hypotheses and concludes the proof.
\end{proof}

\begin{rem}\label{reducible-rem} To compute the difference-differential Galois group $G$ for \eqref{difeq} when there exists exactly one solution $u\in\bar{\mathbb{Q}}(x)$ to \eqref{ric1}, we apply Propositions~\ref{galdiag}, \ref{wrat}, and~\ref{wnotrat} as follows. First, compute the defining equations for the reductive quotient \[\bar{G}:=G/R_u(G)=\left\{\begin{pmatrix}\alpha_\gamma & 0 \\ 0 & \lambda_\gamma\end{pmatrix} \ \middle| \ \gamma\in G\right\},\] which is the $\sigma\delta$-Galois group for the system \eqref{redeq}, as in Proposition~\ref{galdiag} and Remark~\ref{diagonalizable-rem}, with $u_1=u$ and $u_2=v$. In particular, this requires computing the $q$-discrete residues $q$-$\mathrm{dres}\left(\frac{\delta(u)}{u},[\beta]_q,1\right)$ for each $q^\mathbb{Z}$-orbit $[\beta]_q$ with $\beta\in\bar{\mathbb{Q}}^\times$. Note that this will produce all the defining equations for $G$ relating $\alpha$ and $\lambda$ only, and it remains to compute the remaining defining equations for $G$, if there are any.

If $uv^{-1}\neq\frac{\sigma(w)}{w}$ for any $w\in\bar{\mathbb{Q}}(x)$ as in Proposition~\ref{galdiag}(i), then $R_u(G)=\mathbb{G}_a(C)$ by Proposition~\ref{wnotrat}, and therefore there are no more defining equations for $G$. Otherwise, compute such a $w\in\bar{\mathbb{Q}}(x)$, as well as its $q$-discrete residues $q$-$\mathrm{dres}(w,[\beta]_q,j)$ for every $q^\mathbb{Z}$-orbit $[\beta]_q$ and $j\in\mathbb{N}$ (only finitely many of these are non-zero). In this case, \begin{equation}\label{reducible-grels}G\subseteq\bar{G}\times\mathbb{G}_a(C)=\left\{\begin{pmatrix}\alpha & \xi \\ 0 & \alpha\end{pmatrix} \ \middle| \ \begin{pmatrix}\alpha & 0 \\ 0 & \alpha\end{pmatrix}\in\bar{G}, \ \xi\in C \right\},\end{equation} and this containment is proper if and only if there exist $f\in\bar{\mathbb{Q}}(x)$, a linear differential polynomial $\mathcal{L}\in\bar{\mathbb{Q}}\{Y\}_\delta$, and $c\in\bar{\mathbb{Q}}$ as in Proposition~\ref{wrat}(iii).

Let us first compute the defining equations of $G$ in \eqref{reducible-grels} when $q$-$\mathrm{dres}(w,[\beta]_q,j)=0$ for every $q^\mathbb{Z}$-orbit $[\beta]_q$ and $j\in\mathbb{N}$, in which case $q$-$\mathrm{dres}(w,\infty)=:c\neq 0$ and $-w=\sigma(f)-f-c$ for some $f\in\bar{\mathbb{Q}}(x)$ by Proposition~\ref{indsum}, as in Proposition~\ref{wrat}(iii). In this case, $G$ is contained in the subgroup of \eqref{reducible-grels} defined by $\delta\left(\frac{\xi}{\alpha}\right)=0$, and $R_u(G)\subseteq\mathbb{G}_a(C^\delta)$. If $\frac{\delta(u)}{u}=\sigma(\tilde{f})-\tilde{f}$ for some $\tilde{f}\in\bar{\mathbb{Q}}(x)$ as in Proposition~\ref{galdiag}(ii), so that $\delta(\alpha_\gamma)=0$ for every $\gamma\in G$, then $G$ is the subgroup of \eqref{reducible-grels} defined by $\delta(\xi)=0$, and $R_u(G)=\mathbb{G}_a(C^\delta)$. If there exist $\tilde{f}\in\bar{\mathbb{Q}}(x)$ and $0\neq \tilde{c}\in\mathbb{Z}$ as in Proposition~\ref{galdiag}(ii), so that $\delta\left(\frac{\delta(\alpha_\gamma)}{\alpha_\gamma}\right)=0$ for every $\gamma\in G$ but there exists $\gamma\in G$ such that $\delta(\alpha_\gamma)\neq 0$, then $G$ is the subgroup of \eqref{reducible-grels} defined by $\tilde{c}\xi=c\delta(\alpha)$, and $R_u(G)=\{0\}$. If there are no $\tilde{f}\in\bar{\mathbb{Q}}(x)$ and $\tilde{c}\in\mathbb{Z}$ such that $\frac{\delta(u)}{u}=\sigma(\tilde{f})-\tilde{f}+\tilde{c}$, then $G$ is precisely the subgroup of \eqref{reducible-grels} defined by $\delta\left(\frac{\xi}{\alpha}\right)=0$, and $R_u(G)=\mathbb{G}_a(C^\delta)$.

Assuming now that some $q$-discrete residue $q$-$\mathrm{dres}(w,[\beta]_q,j)\neq 0$, let $r\in \mathbb{N}$ be as large as possible such that $q$-$\mathrm{dres}(w,[\beta]_q,r)\neq 0$ for some $q^\mathbb{Z}$-orbit $[\beta]_q$. Write the linear differential polynomial \[\mathcal{L}=\sum_{i=0}^{r-1}c_i\delta^i(Y)\in\bar{\mathbb{Q}}\{Y\}_\delta\] with undetermined coefficients, and decide whether the system of linear equations over $\bar{\mathbb{Q}}$ defined by setting \begin{equation}\label{reducible-op}q\text{-}\mathrm{dres}\left(\mathcal{L}\left(\frac{\delta(u)}{u}\right)-w,[\beta]_q,j\right)=0\end{equation} for every $q^\mathbb{Z}$-orbit $[\beta]_q$ and $1\leq j\leq r$ admits a solution. If there is no solution, then again we have that $R_u(G)=\mathbb{G}_a(C)$ and $G$ is precisely the group in \eqref{reducible-grels}. If there is a solution, then it it is unique and $c_{r-1}\neq 0$. In this case, setting \begin{equation}\label{reducible-rem-c}c:=c_0\cdot q\text{-}\mathrm{dres}\left(\frac{\delta(u)}{u},\infty\right)- q\text{-}\mathrm{dres}\left(w,\infty\right),\end{equation} there exists $f\in\bar{\mathbb{Q}}(x)$ as in Proposition~\ref{wrat}(iii) by Proposition~\ref{indsum}, and $G$ is the subgroup of \eqref{reducible-grels} defined by the corresponding relation stipulated in Proposition~\ref{wrat}, depending on whether the $c\in\bar{\mathbb{Q}}$ defined in \eqref{reducible-rem-c} is zero or not. If $c=0$ then $R_u(G)=0$, and if $c\neq 0$ then $R_u(G)=\mathbb{G}_a(C^\delta)$.
\end{rem}

Since $R_u(G)=\{0\}$ whenever there is not exactly one solution $u\in\bar{\mathbb{Q}}(x)$ to \eqref{ric1} (i.e., either there is no solution or there is more than one solution to \eqref{ric1} in $\bar{\mathbb{Q}}(x)$), we deduce the following result from Remark~\ref{reducible-rem}, which generalizes \cite[Prop.~4.3(2)]{hardouin-singer:2008}.

\begin{cor} \label{unipotent-possibilities} If $G$ is the $\sigma\delta$-Galois group of \eqref{difeq}, then the unipotent radical $R_u(G)$ is either $\{0\}$, $\mathbb{G}_a(C^\delta)$, or $\mathbb{G}_a(C)$.
\end{cor}


\section{Irreducible and imprimitive groups}\label{dihedral-sec}

In this section we will denote $k_1=C(x)$, where $C$ is a $\delta$-closure of $\bar{\mathbb{Q}}$, $\sigma$ denotes the $C$-linear automorphism of $k$ defined by $\sigma(x)=qx$, and $\delta(x)=1$. It will be convenient to use similar notation as that of Section~\ref{hendriks-sec}: fix once and for all $q_2\in\bar{\mathbb{Q}}$ such that $q_2^2=q$, and let $k_2:=C(x_2)$ be the $\sigma\delta$-field extension of $k_1$ defined by setting $x_2^2=x$, $\sigma(x_2)=q_2x_2$, and $\delta(x_2)=\frac{1}{2}x_2$.

Let us now suppose that there are no solutions in $\bar{\mathbb{Q}}(x)$ to the first Riccati equation \eqref{ric1}. According to Proposition~\ref{hendriks-possibilities}, under these conditions the $\sigma$-Galois group $H$ for \eqref{difeq} over $k$ should be irreducible, and $H$ should be imprimitive if and only if one of the following possibilities holds:
\begin{enumerate}
\item there exist two solutions $u_1,u_2\in \bar{\mathbb{Q}}(x_2)\backslash \bar{\mathbb{Q}}(x)$ to the first Riccati equation \eqref{ric1} such that $u_2=\bar{u}_1$ is the Galois conjugate of $u_1$ over $\bar{\mathbb{Q}}(x)$; or 
\item either $a=0$ or else there exists a solution $e\in\bar{\mathbb{Q}}(x)$ to the second Riccati equation \eqref{ric2}; or
\item $a\neq 0$ and there exist two solutions $e_1,e_2\in \bar{\mathbb{Q}}(x_2)\backslash\bar{\mathbb{Q}}(x)$ to the second Riccati equation \eqref{ric2} such that $e_2=\bar{e}_1$ is the Galois conjugate of $e_1$ over $\bar{\mathbb{Q}}(x)$.
\end{enumerate}
\noindent Note that (2) and (3) above are mutually exclusive and together exhaust the possibility that the more compact Proposition~\ref{hendriks-possibilities}(5) holds. We will address each of the possibilites (1), (2), and (3) above in turn, in Sections~\ref{imprimitive-sec:1}, \ref{imprimitive-sec:2}, and \ref{imprimitive-sec:3}, respectively, and establish in each case that $H$ is indeed irreducible and imprimitive in each of these scenarios, as stated in Proposition~\ref{hendriks-possibilities}.

By \cite[Prop.~12.2(1)]{vanderput-singer:1997}, in any case the group of connected components $H/H^\circ$ must be bicyclic. The irreducible and imprimitive algebraic subgroups of $\mathrm{GL}_2(\bar{\mathbb{Q}})$ with bicyclic group of connected components are listed in the following result, which we prove using the classification of the algebraic subgroups of $\mathrm{GL}_2(C)$ developed in \cite{gl2-classification}. In the classification below we denote $\{\pm1\}\times\{\pm1\}$ by $\{\pm1\}^2$ and $\mathbb{G}_m(C)\times\mathbb{G}_m(C)$ by $\mathbb{G}_m(C)^2$.

\begin{lem}\label{imprimitive-classification} If $H$ is an irreducible and imprimitive algebraic subgroup of $\mathrm{GL}_2(C)$ such that $H/H^\circ$ is bicyclic, then $H$ is the subgroup of \begin{equation}\label{imprimitive-form}\{\pm1\}\ltimes\mathbb{G}_m(C)^2=\left\{\begin{pmatrix} \alpha_1 & 0 \\ 0 & \alpha_2\end{pmatrix} \ \middle|\ \alpha_1,\alpha_2\in C^\times\right\}\cup\left\{\begin{pmatrix} 0 & \lambda_1 \\ \lambda_2 & 0\end{pmatrix} \ \middle| \ \lambda_1,\lambda_2\in C^\times\right\}\end{equation} defined by precisely one of the following sets of conditions on $\alpha_1,\alpha_2,\lambda_1$, and $\lambda_2$.
\begin{enumerate}
\item $H=D_m^-$ for some $m\in\mathbb{N}$, defined as the subgroup of \eqref{imprimitive-form} such that $(\alpha_1\alpha_2)^m=1$ and $(\lambda_1\lambda_2)^m=-1$; or
\item $H=D_m^+$ for some $m\in\mathbb{N}$, defined as the subgroup of \eqref{imprimitive-form} such that $(\alpha_1\alpha_2)^m=1$ and $(\lambda_1\lambda_2)^m=1$; or
\item $H=\{\pm1\}^2\ltimes \mathbb{G}_m(C)$, defined as the subgroup of \eqref{imprimitive-form} such that $\alpha_1^2=\alpha_2^2$ and $\lambda_1^2=\lambda_2^2$; or
\item $H=\{\pm1\}\ltimes\mathbb{G}_m(C)^2$ as in \eqref{imprimitive-form}, with no other conditions on $\alpha_1,\alpha_2,\lambda_1$, and $\lambda_2$.
\end{enumerate}
\end{lem}

\begin{proof} The algebraic subgroups $H\subseteq \mathrm{GL}_2(\bar{\mathbb{Q}})$ are classified in \cite{gl2-classification} according to their projective image $\bar{H}\subseteq \mathrm{PGL}_2(\bar{\mathbb{Q}})$. Since $H$ is irreducible and imprimitive with bicyclic group of connected components, it is an infinite non-commutative subgroup of \eqref{imprimitive-form}, and therefore its projective image is either $\bar{H}=D_n$, the dihedral group of order $2n$ for some $n\geq 2$, or else $\bar{H}=\bar{D}_\infty$, the projective image of \[D_\infty=\left\{\begin{pmatrix} \alpha & 0 \\ 0 & \alpha^{-1}\end{pmatrix} \ \middle| \ \alpha\in\bar{\mathbb{Q}}^\times\right\}\cup\left\{\begin{pmatrix} 0 & -\lambda \\ \lambda^{-1} & 0\end{pmatrix} \ \middle| \ \lambda\in\bar{\mathbb{Q}}^\times\right\}.\] 

If $\bar{H}=D_n$ then $D_n$ must be commutative, since the algebraic quotient map $H\rightarrow \bar{H}$ factors through $H/H^\circ$, which we are assuming is abelian, and therefore $n=2$ (corresponding to $D_n\simeq K_4$, the Klein four-group) in this case. By \cite[Thm.~4]{gl2-classification}, the \emph{minimal} subgroups (see \cite[\S2]{gl2-classification} for the definition) of $\mathrm{GL}_2(\bar{\mathbb{Q}})$ having projective image $D_2$ are ${D}_{2,\ell}$ for some $\ell\in\mathbb{Z}_{\geq 0}$, where \[{D}_{2,\ell}:=\left\langle\zeta_{2^{\ell+1}}\begin{pmatrix} i & 0 \\ 0 & -i\end{pmatrix},\begin{pmatrix} 0 & i \\ i & 0\end{pmatrix}\right\rangle\] and $\zeta_{2^{\ell+1}}$ denotes a primitve $(2^{\ell+1})$-th root of unity. Therefore the only infinite subgroups $H\subseteq \mathrm{GL}_2(\bar{\mathbb{Q}})$ having projective image $D_2$  are given by $\bar{\mathbb{Q}}^\times\cdot {D}_{2,\ell}$, which are all equal to $\pi^{-1}(D_2)$, where $\pi:\mathrm{GL}_2(\bar{\mathbb{Q}})\rightarrow\mathrm{PGL}_2(\bar{\mathbb{Q}})$ is the projection map. Finally, note that $\bar{\mathbb{Q}}^\times \cdot {D}_{2,\ell}$ for any $\ell\in\mathbb{Z}_{\geq 0}$ is precisely the subgroup of \eqref{imprimitive-form} defined by the conditions in item~(3): $\alpha_1^2=\alpha_2^2$ and $\lambda_1^2=\lambda_2^2$.

If $\bar{H}=\bar{D}_\infty$, then either $H=\{\pm 1\}\ltimes\mathbb{G}_m(\bar{\mathbb{Q}})$ in \eqref{imprimitive-form} as in item~(4), or else $H=\mu_n\cdot D_{\infty,\ell}$ for some $\ell\in\mathbb{Z}_{\geq 0}$ and some $n\in\mathbb{N}$, where $\mu_n$ denotes the group of $n$-th roots of unity, and \[D_{\infty,\ell}:=\left\langle\left\{\begin{pmatrix}\alpha & 0 \\ 0 & \alpha^{-1}\end{pmatrix} \ \middle| \ \alpha\in\bar{\mathbb{Q}}^\times\right\},\begin{pmatrix} 0 & \zeta_{2^{\ell+1}} \\ \zeta_{2^{\ell+1}} & 0\end{pmatrix}\right\rangle,\] where again $\zeta_{2^{\ell+1}}$ is a primitive $(2^{\ell+1})$-th root of unity, since by \cite[Thm.~4]{gl2-classification} the $D_{\infty,\ell}$ are all the minimal subgroups of $\mathrm{GL}_2(\bar{\mathbb{Q}})$ with projective image $\bar{D}_\infty$. All of these groups have the property that $H/H^\circ$ is bicyclic. It remains to show that for any $n\in\mathbb{N}$ and $\ell\in\mathbb{Z}_{\geq 0}$ the group $\mu_n\cdot D_{\infty,\ell}$ is one of the groups described by the conditions in either item~(1) or item~(2). Let us write $\Delta_{n,\ell}:=\left\{\left(\begin{smallmatrix}\alpha_1 & 0 \\ 0 & \alpha_2\end{smallmatrix}\right) \ \middle| \ \alpha_1,\alpha_2\in\bar{\mathbb{Q}}^\times\right\}\cap (\mu_n\cdot D_{\infty,\ell})$, the group of all diagonal matrices contained in $\mu_n\cdot D_{\infty,\ell}$, and $\nabla_{n,\ell}:=\left\{\left(\begin{smallmatrix} 0 & \lambda_1\\ \lambda_2 & 0\end{smallmatrix}\right) \ \middle| \ \lambda_1,\lambda_2\in\bar{\mathbb{Q}}^\times\right\}\cap (\mu_n\cdot D_{\infty,\ell})$ for the complementary coset of $\Delta_{n,\ell}$ in $\mu_n\cdot D_{\infty,\ell}$ consisting of all the antidiagonal matrices contained in $\mu_n\cdot D_{\infty,\ell}$. Then we see that $\nabla_{n,\ell}=\zeta_{2^{\ell+1}}\cdot\Delta_{n,\ell}\cdot\left(\begin{smallmatrix} 0 & 1 \\ 1 & 0\end{smallmatrix}\right)$, and $\Delta_{n,\ell}=\langle \zeta_n,\zeta_{2^\ell}\rangle\cdot\left\{\left(\begin{smallmatrix} \alpha & 0 \\ 0 & \alpha^{-1} \end{smallmatrix}\right) \ \middle| \ \alpha\in\bar{\mathbb{Q}}^\times\right\}$. Therefore, 
\[\Delta_{n,\ell}=\left\{\begin{pmatrix}\alpha_1 & 0 \\ 0 & \alpha_2\end{pmatrix} \ \middle| \ \alpha_1,\alpha_2\in\bar{\mathbb{Q}}^\times, \ (\alpha_1\alpha_2)^m=1\right\},\quad\text{where}\quad m:=\begin{cases}\frac{1}{2}\mathrm{lcm}(n,2^\ell) \ \text{if} \ \ell\geq 1;\\ \frac{n}{2} \ \text{if} \ \ell=0 \ \text{and} \ 2|n;\\ n \ \text{if} \ \ell=0 \ \text{and} \ 2\nmid n;\end{cases}\]
because $(\langle\zeta_n,\zeta_{2^\ell}\rangle)^2=\langle\zeta_m\rangle$ with $m$ defined as above. Since \[\nabla_{n,\ell}=\left\{\begin{pmatrix} 0 & \alpha_1\zeta_{2^{\ell+1}} \\ \alpha_2\zeta_{2^{\ell+1}} & 0\end{pmatrix} \ \middle| \ \alpha_1,\alpha_2\in\bar{\mathbb{Q}}^\times, \ (\alpha_1\alpha_2)^m=1\right\},\] we have that $\mu_n\cdot D_{\infty,\ell}$ is the group described in item~(2) if and only if $2^{\ell}|m$ (which occurs precisely when either $\ell=0$ or else $\ell\geq 1$ and $2^{\ell+1}|n$), and $\mu_n\cdot D_{\infty,\ell}$ is the group described in item~(1) otherwise, since for $\ell\geq 1$ we always have that $2^{\ell-1}|m$, and therefore $(\zeta_{2^{\ell+1}}^2\alpha_1\alpha_2)^m=(\zeta_{2^\ell}\alpha_1\alpha_2)^m=-1$, precisely when $2^\ell\nmid m$, $2^{\ell-1}|m$ (with $\ell\geq 1$), and $(\alpha_1\alpha_2)^m=1$.\end{proof}

\begin{rem}\label{imprimitive-diag-det-rem} Given an irreducible and imprimitive algebraic subgroup $H\subseteq \mathrm{GL}_2(\bar{\mathbb{Q}})$ such that the group of connected components $H/H^\circ$ is bicyclic, we can uniquely identify it among the possibilities listed in Lemma~\ref{imprimitive-classification} by the knowledge of two auxiliary groups: \[\Delta(H):=\left\{\begin{pmatrix} \alpha_1 & 0 \\ 0 & \alpha_2\end{pmatrix} \ \middle| \ \alpha_1,\alpha_2\in\bar{\mathbb{Q}}^\times\right\}\cap H \qquad \text{and} \qquad \mathrm{det}(H)=\{\mathrm{det}(h) \ | \ h\in H\}\subseteq\mathbb{G}_m(\bar{\mathbb{Q}}),\] respectively the subgroup of diagonal matrices in $H$ and the image of $H$ under the determinant map. Indeed, $\Delta(H)=\mathbb{G}_m(\bar{\mathbb{Q}})^2$ if and only if $H=\{\pm 1\}\ltimes\mathbb{Q}_m(\bar{\mathbb{Q}})^2$ as in Lemma~\ref{imprimitive-classification}(4); $\Delta(H)=\left\{\left(\begin{smallmatrix}\alpha_1 & 0 \\ 0 & \alpha_2\end{smallmatrix}\right) \ \middle| \ \alpha_1,\alpha_2\in\bar{\mathbb{Q}}^\times, \ \alpha_1^2=\alpha_2^2\right\}$ if and only if $H=\{\pm1\}^2\ltimes\mathbb{G}_m(\bar{\mathbb{Q}})$ is as in Lemma~\ref{imprimitive-classification}(3); and $\Delta(H)=\left\{\left(\begin{smallmatrix}\alpha_1 & 0 \\ 0 & \alpha_2\end{smallmatrix}\right) \ \middle| \ \alpha_1,\alpha_2\in\bar{\mathbb{Q}}^\times, \ (\alpha_1\alpha_2)^m\right\}$ for some $m\in\mathbb{N}$ if and only if $H$ is one of the groups $D_m^-$ or $D_m^+$ described respectively in items (1) or (2) of Lemma~\ref{imprimitive-classification}. To decide between these cases, note that $\mathrm{det}(H)=\langle \alpha_1\alpha_2, -\lambda_1\lambda_2\rangle$ has $\mathrm{det}(H)^m=\langle(-1)^m(\lambda_1\lambda_2)^m\rangle$; hence, if $m$ is even, then $H=D_m^-$ if and only if $\mathrm{det}(H)=\mu_{2m}$ and $H=D_m^+$ if and only if $\mathrm{det}(H)=\mu_m$; and if $m$ is odd, then $H=D_m^-$ if and only if $\mathrm{det}(H)=\mu_m$ and $H=D_m^+$ if and only if $\mathrm{det}(H)=\mu_{2m}$.\end{rem}

\subsection{Irreducible and imprimitive (1): diagonalizable over the quadratic extension}\label{imprimitive-sec:1}

Supposing there are no solutions to \eqref{ric1} in $\bar{\mathbb{Q}}(x)$, but there are two solutions $u,\bar{u}\in\bar{\mathbb{Q}}(x_2)$ to \eqref{ric1}, Galois-conjugate over $\bar{\mathbb{Q}}(x)$, the system \eqref{mateq} \begin{align} \sigma(Y)&=\begin{pmatrix}0 & 1 \\ -b & -a\end{pmatrix} Y\quad &\text{with fundamental solution matrix} \qquad Y&=\begin{pmatrix} y_1 & y_2 \\ \sigma(y_1) & \sigma(y_2)\end{pmatrix}; \ \text{and}\nonumber \\ \sigma(Z)&=\begin{pmatrix} u & 0 \\ 0 & \bar{u}\end{pmatrix}Z\quad&\text{with fundamental solution matrix}\qquad Z&=\begin{pmatrix} z_1 & 0 \\ 0 & z_2\end{pmatrix}\label{quadratic-reducible-system}\end{align} are equivalent over $\bar{\mathbb{Q}}(x_2)$ via the gauge transformation $Z=TY$, where (cf.~Remark~\ref{imprimitive-gauge})\begin{equation}\label{imprimitive-1-gauge}T:=\begin{pmatrix}\frac{\bar{u}}{u-\bar{u}} & \frac{-1}{u-\bar{u}} \\ \frac{u}{u-\bar{u}} & \frac{-1}{u-\bar{u}}\end{pmatrix}\in\mathrm{GL}_2(\bar{\mathbb{Q}}(x_2)).\end{equation} Let us write $S_2=k_2[Y,\mathrm{det}(Y)^{-1}]=k_2[z_1,z_2,(z_1z_2)^{-1}]$ for the $\sigma$-PV ring for \eqref{mateq} (or equivalently for \eqref{quadratic-reducible-system}) over $k_2$. Then $S_1=k_1[Y,\mathrm{det}(Y)^{-1}]\subset S_2$ is a $\sigma$-PV ring for \eqref{mateq} over $k_1$. Let us also write $H_i=\mathrm{Gal}_\sigma(S_i/k_i)$ for $i=1,2$, and $\tilde{H}=\mathrm{Gal}_\sigma(S_2/k_1)$. Since \eqref{quadratic-reducible-system} is a diagonal system, the group $H_2$ is diagonalizable. By Proposition~\ref{fiber-product}, \[\tilde{H}\simeq H_1\times_{\mu_m}\mu_2,\] where $m\in\{1,2\}$ is determined by the intersection $S_1\cap k_2=k_m$ inside $S_2$, and $H_2$ is an index-$m$ subgroup of $H_1$. We claim that any $\tilde{\tau}\in\tilde{H}$ such that $\tilde{\tau}(x_2)=-x_2$ has the property that $\tau:=\tilde{\tau}|_{S_1}\in H_1$ is given by an anti-diagonal matrix. From this it will follow that $H_2$ has index exactly $2$ in $H_1$, and $H_1=H_2\cup H_2\cdot\tau$ is irreducible and imprimitive, as claimed in Proposition~\ref{hendriks-possibilities}(4).

To see this, let $M_\tau\in\mathrm{GL}_2(C)$ such that $\tau(Y)=YM_\tau$. Then for the gauge transformation $T$ given in \eqref{imprimitive-1-gauge} we see that \[\tilde{\tau}(Z)=\tilde{\tau}(TY)=\bar{T}YM_\tau=\left(\begin{smallmatrix} 0 & 1 \\ 1 & 0\end{smallmatrix}\right)ZM_\tau.\] On the other hand, we see that $\sigma(\tilde{\tau}(z_1))=\tilde{\tau}(\sigma(z_1))=\tilde{\tau}(uz_1)=\bar{u}\tilde{\tau}(z_1)$, and therefore $\tilde{\tau}(z_1)=\lambda_2 z_2$ for some $\lambda_2\in C^\times$. A similar computation shows that $\tilde{\tau}(z_2)=\lambda_1z_1$ for some $\lambda_1\in C^\times$. From this it follows that \[\tilde{\tau}(Z)=\begin{pmatrix}\tilde{\tau}(z_1) & 0 \\ 0 & \tilde{\tau}(z_2)\end{pmatrix}=\begin{pmatrix}\lambda_2z_2 & 0 \\ 0 & \lambda_1z_1\end{pmatrix}=\begin{pmatrix}0 & 1 \\ 1 & 0 \end{pmatrix}Z\begin{pmatrix} 0 & \lambda_1 \\ \lambda_2 & 0 \end{pmatrix}.\] Hence $M_\tau=\left(\begin{smallmatrix} 0 & \lambda_1 \\ \lambda_2 & 0 \end{smallmatrix}\right)$, as we wanted to show.

\begin{rem}\label{imprimitive-1-h-rem} Having established that the $\sigma$-Galois group $H_1$ for \eqref{mateq} over $k_1$ is indeed irreducible and imprimitive as claimed in Proposition~\ref{hendriks-possibilities}(4), we can compute this $H_1$ from among the possibilities listed in Lemma~\ref{imprimitive-classification} as explained in Remark~\ref{imprimitive-diag-det-rem}, by determining the subgroup $\Delta(H_1)$ of diagonal matrices in $H_1$, and the group $\mathrm{det}(H_1)\subseteq\mathbb{G}_m(C)$.

Since $\mathrm{det}(H_1)$ is the $\sigma$-Galois group for the system $\sigma(y)=by$, we see that $\mathrm{det}(H_1)=\mu_m$ if and only if $m\in\mathbb{N}$ is the smallest positive integer such that $b^m=\frac{\sigma(f)}{f}$ for some $f\in\bar{\mathbb{Q}}(x)$, and if there is no such $m$ then $\mathrm{det}(H_1)=\mathbb{G}_m(C)$.

Since $\Delta(H_1)=H_2$ is the $\sigma$-Galois group for \eqref{quadratic-reducible-system} over $k_2$, we can compute the defining equations for \[\Delta(H_1)\subseteq \left\{\begin{pmatrix}\alpha_1 & 0 \\ 0 & \alpha_2\end{pmatrix} \ \middle| \ \alpha_1,\alpha_2\in C^\times\right\}\] as follows:
\begin{enumerate}
\item $(\alpha_1\alpha_2)^m=1$ if and only if $(u\bar{u})^m=\frac{\sigma(f)}{f}$ for some $f\in\bar{\mathbb{Q}}(x_2)^\times$ (and in this case $H_1$ is $D_m^-$ or $D_m^+$);
\item $\alpha_1^2=\alpha_2^2$ if and only if $\left(\frac{u}{\bar{u}}\right)^2=\frac{\sigma(f)}{f}$ for some $f\in\bar{\mathbb{Q}}(x_2)$ (and in this case $H_1=\{\pm1\}^2\ltimes\mathbb{G}_m(C)$) ;
\item if none of these possibilities holds, then $\Delta(H_1)=\mathbb{G}_m(C)^2$ (and in this case $H_1=\{\pm1\}\ltimes\mathbb{G}_m(C)^2$).
\end{enumerate} \end{rem}

The computation of the $\sigma\delta$-Galois group $G_1$ for \eqref{mateq} over $k_1$, assuming that the corresponding $\sigma$-Galois group $H_1$ has already been computed as in Remark~\ref{imprimitive-1-h-rem}, will by achieved analogously in the following result, by studying the $\sigma\delta$-Galois group $G_2$ for \eqref{mateq} over $k_2$.

\begin{prop}\label{imprimitive-prop-1} Suppose there are no solutions to \eqref{ric1} in $\bar{\mathbb{Q}}(x)$, and let $u,\bar{u}\in \bar{\mathbb{Q}}(x_2)$ satisfy \eqref{ric1}. Then $G_1$ is the subgroup of \begin{equation}\label{imprimitive-eq-1} \{\pm1\}\ltimes\mathbb{G}_m(C)^2=\left\{\begin{pmatrix}\alpha _1& 0 \\ 0 & \alpha_2\end{pmatrix}, \begin{pmatrix}0 & \lambda_1 \\ \lambda_2 & 0 \end{pmatrix} \ \middle| \ \alpha_1\alpha_2\neq 0, \ \lambda_1\lambda_2\neq 0\right\}\end{equation} defined by the following conditions on $\alpha_1,\alpha_2,\lambda_1$, and $\lambda_2$.
\begin{enumerate}
\item If $H_1=D_m^-$ as in Lemma~\ref{imprimitive-classification}(1) or $H_1=D_m^+$ as in Lemma~\ref{imprimitive-classification}(2), then $G_1=H_1$.
\item If $H_1=\{\pm1\}^2\ltimes\mathbb{G}_m(C)$ as in Lemma~\ref{imprimitive-classification}(3), then:
\begin{enumerate}
\item there exist $0\neq c\in 2\mathbb{Z}$ and $g\in\bar{\mathbb{Q}}(x)$ such that $\frac{\delta(b)}{b}=\sigma(g)-g+c$ if and only if $G_1$ is the subgroup of $H_1$ defined by $\delta\left(\frac{\delta(\alpha_1)}{\alpha_1}+\frac{\delta(\alpha_2)}{\alpha_2}\right)=0=\delta\left(\frac{\delta(\lambda_1)}{\lambda_1}+\frac{\delta(\lambda_2)}{\lambda_2}\right)$; 
\item otherwise, $G_1=H_1$.
\end{enumerate}
\item If $H_1=\{\pm1\}\ltimes\mathbb{G}_m(C)^2$ as in Lemma~\ref{imprimitive-classification}(4), then:
\begin{enumerate}
\item there exist $c\in 2\mathbb{Z}$ and $g\in\bar{\mathbb{Q}}(x)$ such that $\frac{\delta(b)}{b}=\sigma(g)-g+c$ if and only if $\delta\left(\frac{\delta(\alpha_1)}{\alpha_1}+\frac{\delta(\alpha_2)}{\alpha_2}\right)=0=\delta\left(\frac{\delta(\lambda_1)}{\lambda_1}+\frac{\delta(\lambda_2)}{\lambda_2}\right)$; moreover, $c=0$ if and only if $\delta(\alpha_1\alpha_2)=0=\delta(\lambda_1\lambda_2)$;
\item otherwise, $G_1=H_1$.
\end{enumerate}
\end{enumerate}
\end{prop}

\begin{proof}

Since they systems \eqref{mateq} and \eqref{quadratic-reducible-system} are equivalent over $k_2$, and the latter system is diagonal, we can compute $G_2$ with Proposition~\ref{galdiag} and Remark~\ref{diagonalizable-rem}, but with a small caveat. Namely, after replacing $\delta$ with $\delta_2:=2\delta$, we see that $k_2$ as a $\sigma\delta_2$-field behaves just as $k_1$: $\sigma(x_2)=q_2x_2$ and $\delta_2(x_2)=x_2$. Thus we may compute the $\delta_2$-algebraic group $G_2\subseteq\mathbb{G}_m(C)^2$ over $k_2$ using the procedure described in Remark~\ref{diagonalizable-rem} exactly as stated there, and then simply replace every instance of $\delta_2$ in the defining equations for $G_2$ with $\frac{1}{2}\delta$ a posteriori. But since the system \eqref{quadratic-reducible-system} has such a special form, not every possibility listed in Proposition~\ref{galdiag} may occur.

We saw in Remark~\ref{diagonalizable-rem} that $G_2$ is a proper subgroup of $\mathbb{G}_m(C)^2$ if and only if there exist: $m_1,m_2\in\mathbb{Z}$, not both zero and with $\mathrm{gcd}(m_1,m_2)=1$; $c\in \mathbb{Z}$; and $g\in\bar{\mathbb{Q}}(x_2)$, such that \begin{equation}\label{imprimitive-symmetry}m_1\frac{\delta_2(u)}{u}+m_2\frac{\delta_2(\bar{u})}{\bar{u}}=\sigma(g)-g+c \qquad \Longleftrightarrow \qquad m_2\frac{\delta_2(u)}{u}+m_1\frac{\delta_2(\bar{u})}{\bar{u}}=\sigma(\bar{g})-\bar{g}+c.\end{equation} Let us consider the submodule $M\subseteq\mathbb{Z}^2$ generated by relatively prime pairs $(m_1,m_2)$ such that there exist $g\in\bar{\mathbb{Q}}(x_2)$ and $c\in\mathbb{Z}$ satisfying the above conditions. Then, as we saw in Remark~\ref{diagonalizable-rem}, either $M=\{0\}$ is trivial; or $M=\mathbb{Z}\cdot (m_1,m_2)$ is infinite cyclic; or $M=\mathbb{Z}^2$. Moreover, $M=\{0\}$ is trivial if and only if the $\sigma\delta$-Galois group $G_2$ for \eqref{quadratic-reducible-system} is all of $\mathbb{G}_m(C)^2$. In  this case we must have $G_1=H_1=\{\pm 1 \}\ltimes\mathbb{G}_m(C)^2$, because $G_1$ is Zariski-dense in $H_1$ by Proposition~\ref{dense}, and therefore $G_1$ contains at least one anti-diagonal matrix, whence it contains all anti-diagonal matrices.

From now on we assume that $M$ is not trivial. It follows from \eqref{imprimitive-symmetry} that at least one of $(1,1)$ or $(1,-1)$ belongs to $M$. In any case it is useful to observe that \[q_2\text{-}\mathrm{dres}\left(\frac{\delta_2(u)}{u},\infty\right)=d=q_2\text{-}\mathrm{dres}\left(\frac{\delta_2(\bar{u})}{\bar{u}},\infty\right),\] where $d\in\mathbb{Z}$ is the common degree of $u$ and $\bar{u}$ considered as rational functions in $x_2$. Therefore, $(1,-1)\in M$ if and only if $\delta(\alpha_1\alpha_2)=0$ for every $\left(\begin{smallmatrix}\alpha_1 & 0 \\ 0 & \alpha_2\end{smallmatrix}\right)\in G_2$.

We claim that actually $(1,-1)\in M$ if and only if $H_1=\{\pm 1\}^2\ltimes\mathbb{G}_m(C)$ as in Lemma~\ref{imprimitive-classification}(3). As explained in Remark~\ref{imprimitive-1-h-rem}, $H_1=\{\pm 1\}^2\ltimes\mathbb{G}_m(C)$ if and only if there exists $f\in\bar{\mathbb{Q}}(x_2)^\times$ such that $\left(\frac{u}{\bar{u}}\right)^2=\frac{\sigma(f)}{f}$, which in turn implies that \[\frac{\delta_2(u)}{u}-\frac{\delta_2(\bar{u})}{\bar{u}}=\sigma\left(\frac{1}{2}\frac{\delta_2(f)}{f}\right)-\frac{1}{2}\frac{\delta_2(f)}{f}.\] Thus if $H=\{\pm1\}\ltimes\mathbb{G}_m(C)$ then $(1,-1)\in M$. To establish the opposite implication, let us study the \emph{reduced form} of $u$: there exists $v\in\bar{\mathbb{Q}(x_2)}$ such that $u\frac{\sigma(v)}{v}=ex_2^n\frac{p_1}{p_2}$, where $e\in\bar{\mathbb{Q}}^\times$ is such that if $e\in q_2^\mathbb{Z}$ then $e=1$, $n\in\mathbb{Z}$ is arbitrary, and $p_1,p_2\in\bar{\mathbb{Q}}[x_2]$ are monic such that $\mathrm{gcd}(x_2,p_1)=\mathrm{gcd}(x_2,p_2)=\mathrm{gcd}(p_1,\sigma^m(p_2))=1$ for every $m\in\mathbb{Z}$. We say that $e x_2^n\frac{p_1}{p_2}$ is the \emph{reduced form} of $u$. We then see that the reduced form of $\bar{u}$ is $(-1)^n e x_2^n\frac{\bar{p}_1}{\bar{p}_2}$. Although it need not be the case that the reduced form of $\frac{u}{\bar{u}}$ is exactly \[(-1)^n\frac{p_1\bar{p}_2}{p_2\bar{p}_1},\] (because it is possible for $\mathrm{gcd}(p_1,\sigma^m(\bar{p}_2))\neq 1$ for some $m\in\mathbb{Z}$), we see that in any case the reduced form of $\frac{u}{\bar{u}}$ is similarly given by \[(-1)^n\frac{\tilde{p}_1}{\tilde{p}_2}\] for some $\tilde{p}_1,\tilde{p}_2\in\bar{\mathbb{Q}}[x_2]$ monic and such that $\mathrm{gcd}(x_2,\tilde{p}_1)=\mathrm{gcd}(x_2,\tilde{p}_2)=\mathrm{gcd}(\tilde{p}_1,\sigma^m(\tilde{p}_2))=1$ for every $m\in\mathbb{Z}$. But then we see that if, say, $\tilde{p}_1\neq 1$, then there exists $\beta\in\bar{\mathbb{Q}}^\times$ such that $\tilde{p}_1(\beta)=0$, and we have that \[q_2\text{-}\mathrm{dres}\left(\frac{\delta_2(\tilde{p}_1)}{\tilde{p}_1},[\beta]_{q_2},1\right)\neq 0= q_2\text{-}\mathrm{dres}\left(\frac{\delta_2(\tilde{p}_2)}{\tilde{p}_2},[\beta]_{q_2},1\right),\] and similarly if we assume instead that $\tilde{p}_2\neq 1$. Therefore, if either $\tilde{p}_1\neq 1$ or $\tilde{p}_2\neq 1$, it is impossible to have $(1,-1)\in M$. Or in other words, if $(1,-1)\in M$ then $\frac{u}{\bar{u}}=(-1)^n\frac{\sigma(\tilde{f})}{\tilde{f}}$ for some $\tilde{f}\in\bar{\mathbb{Q}}(x_2)^\times$. But in this case we then see that $n$ must be odd, for otherwise we would have that $\alpha_1=\alpha_2$ for every $\left(\begin{smallmatrix}\alpha_1 & 0 \\ 0 & \alpha_2\end{smallmatrix}\right)\in G_2$, and since $G_2$ is Zariski-dense in $H_2$ the same relation would be satisfied by every diagonal matrix in $H_1$, but this does not occur for any of the possibilities for $H_1$ listed in Lemma~\ref{imprimitive-classification}. This conlcudes the proof that $(1,-1)\in M$ if and only if $H_1=\{\pm\}\ltimes\mathbb{G}_m(C)$ as in Lemma~\ref{imprimitive-classification}(3).

In case we do have $(1,-1)\in M$, we must decide whether $M=\mathbb{Z}\cdot (1,-1)$ or $M=\mathbb{Z}^2$. We have that $M=\mathbb{Z}\cdot (1,-1)$ if and only if $G_2=H_2$, which implies that $G_1=H_1$. On the other hand, we have $M=\mathbb{Z}^2$ if and only if $(1,1)\in M$ also, i.e., \eqref{imprimitive-symmetry} is satisfied with $m_1=1=m_2$. But then after adding those two equations together we see that there exists $f\in\bar{\mathbb{Q}}(x)$ (not just in $\bar{\mathbb{Q}}(x_2)$), such that \[\frac{\delta(b)}{b}=\frac{\delta(u)}{u}+\frac{\delta(\bar{u})}{\bar{u}}+\sigma\left(\frac{\delta(w)}{w}\right)-\frac{\delta(w)}{w}=\sigma(f)-f+2c.\] Indeed, writing $u-\bar{u}=x_2w$ with $w\in\bar{\mathbb{Q}}(x)^\times$ and $f:=\frac{1}{2}(g+\bar{g})+\frac{\delta(w)}{w}\in\bar{\mathbb{Q}}(x)$, where $g\in\bar{\mathbb{Q}}(x_2)$ and $c\in\mathbb{Z}$ are as in \eqref{imprimitive-symmetry}, the above equation results from comparing determinants in $\sigma(T)AT^{-1}=\left(\begin{smallmatrix}u & 0 \\ 0 & \bar{u}\end{smallmatrix}\right)$ with $T$ as in \eqref{imprimitive-1-gauge}. Furthermore, in this case we must have $c\neq 0$, for otherwise we would have that $G_1\subseteq\mathrm{GL}_2(C^\delta)$ is differentially constant, which by \cite[Thm.~3.7(ii)]{arreche-singer:2016} would imply that $G_1$ is commutative. But this is impossible, since $G_1$ is Zariski-dense in $H_1$ by Proposition~\ref{dense}, so $H_1$ would have to be commutative also, yielding a contradiction. Thus, $G_2$ is a proper subgroup of $H_2$ if and only if \[G_2=\left\{\begin{pmatrix} \alpha & 0 \\ 0 & \alpha\end{pmatrix}, \ \begin{pmatrix} \alpha & 0 \\ 0 & -\alpha\end{pmatrix} \ \middle| \ \alpha\in C^\times \ \text{with} \ \delta\left(\frac{\delta(\alpha)}{\alpha}\right)=0\right\}.\] Since for any $\left(\begin{smallmatrix} 0 & \pm\lambda \\ \lambda & 0\end{smallmatrix}\right)\in G_1$ we have that $\lambda^2\left(\begin{smallmatrix} 1 & 0 \\ 0 & 1\end{smallmatrix}\right)\in G_2$, we see that $\delta\left(\frac{\delta(\lambda)}{\lambda}\right)=0$ also, concluding the proof of item~(2).

It remains to show that the statements in items~(1) and (3) are correct when $M=\mathbb{Z}\cdot (1,1)$. If $H_1=D_m^-$ or $H_1=D_m^+$, then $G_2=H_2$ and therefore $G_1=H_1$. This establishes item~(1). Finally, supposing $H_1=\{\pm 1\}\ltimes\mathbb{G}_m(C)^2$ and $M=\mathbb{Z}\cdot (1,1)$, the arguments above and in Remark~\ref{diagonalizable-rem} show that this occurs if and only if $\frac{\delta(b)}{b}=\sigma(f)-f+2c$, if and only if $\mathrm{det}(G_2)\subseteq\left\{\alpha\in C^\times \ \middle| \ \delta\left(\frac{\delta(\alpha)}{\alpha}\right)\right\},$ with equality if and only if $c\neq 0$, and moreover $c=0$ if and only if $\mathrm{det}(G_2)=\{\alpha\in C^\times \ | \ \delta(\alpha)=0\}$. Since $G_2$ has index $2$ in $G_1$, $\mathrm{det}(G_2)$ has index at most $2$ in $\mathrm{det}(G_1)$; but since $\mathrm{det}(G_2)$ is divisible in either case, we see that $\mathrm{det}(G_2)=\mathrm{det}(G_1)$, concluding the proof of item~(3). \end{proof}

\subsection{Irreducible and imprimitive (2): rational system of imprimitivity}\label{imprimitive-sec:2}

Supposing there are no solutions to \eqref{ric1} in $\bar{\mathbb{Q}}(x_2)$, and either $a=0$ or there exists a solution $e\in\bar{\mathbb{Q}}(x)$ to \eqref{ric2}, we proceed as follows. The non-existence of solutions to \eqref{ric1} in $k_2$ implies there are no solutions in $k_\infty$ either, which in turn implies that the $\sigma$-Galois group $H_\infty$ for \eqref{mateq} over $k_\infty$ is irreducible, and since $H_\infty\subseteq H_1$, the $\sigma$-Galois group for \eqref{mateq} over $k_1$, we then have that $H_1$ must be irreducible also.

The system \eqref{mateq} in this case is equivalent to \begin{equation}\label{std-imprimitive-eq}\sigma(Y)=\begin{pmatrix} 0 & 1 \\ -r & 0\end{pmatrix} Y,\end{equation} for some $r\in\bar{\mathbb{Q}}(x)$ as we saw in \S\ref{hendriks-sec}, which implies that $H_\infty$ is imprimitive. Since $H_\infty$ has finite index in $H_1$, the classification of algebraic subgroups of $\mathrm{GL}_2(C)$ from \cite{gl2-classification} then implies that $H_1$ must also be imprimitive, and therefore $H_1$ must be one of the irreducible imprimitive subgroups of $\mathrm{GL}_2(C)$ with bicyclic group of connected components listed in Lemma~\ref{imprimitive-classification}.

\begin{rem}\label{imprimitive-2-h-rem}In this case, we can compute the $\sigma$-Galois group $H_1$ for \eqref{std-imprimitive-eq} over $k_1$ from among the possibilities listed in Lemma~\ref{imprimitive-classification} with the aid of Remark~\ref{imprimitive-diag-det-rem} by computing the diagonal subgroup $\Delta(H_1)$ and the image of the determinant $\mathrm{det}(H_1)$ as follows. As before, $\mathrm{det}(H_1)=\mu_m$, the group of $m$-th roots of unity, if and only if $m$ is the smallest positive integer such that $b^m=\frac{\sigma(f)}{f}$ for some $f\in\bar{\mathbb{Q}}(x)^\times$; if there is no such $m$, then $\mathrm{det}(H_1)=\mathbb{G}_m(C)$. On the other hand, $\Delta(H_1)$ is precisely the $\sigma^2$-Galois group for the system \begin{equation}\label{imprimitive-2-diag-system} \sigma^2(Z)=\begin{pmatrix}-r & 0 \\ 0 & -\sigma(r)\end{pmatrix}Z\end{equation} over $k_1$, which we can compute as in Proposition~\ref{galdiag} and Remark~\ref{diagonalizable-rem} by considering $k_1$ as a $\sigma^2$-field. We see that \[\Delta(H_1)\subseteq \left\{\begin{pmatrix}\alpha_1 & 0 \\ 0 & \alpha_2\end{pmatrix} \ \middle| \ \alpha_1,\alpha_2\in C^\times\right\}\] is the subgroup defined by the following conditions on $\alpha_1$ and $\alpha_2$:
\begin{enumerate}
\item $(\alpha_1\alpha_2)^m=1$ if and only if $m$ is the smallest positive integer such that $(r\sigma(r))^m=\frac{\sigma^2(f)}{f}$ for some $f\in\bar{\mathbb{Q}}(x)^\times$;
\item otherwise $\Delta(H_1)=\mathbb{G}_m(C)^2$.
\end{enumerate}
The omission of the possibility that $H_1=\{\pm 1\}^2\ltimes \mathbb{G}_m(C)$ as in Lemma~\ref{imprimitive-classification}(3) is deliberate. This is impossible under the present assumptions because $\alpha_1^2=\alpha_2^2$ for every $\left(\begin{smallmatrix}\alpha_1 & 0 \\ 0 & \alpha_2\end{smallmatrix}\right)\in\Delta(H_1)$ if and only if $\left(\frac{\sigma(r)}{r}\right)^2=\frac{\sigma^2(f)}{f}$ for some $f\in\bar{\mathbb{Q}}(x)^\times$. But if we let $v\in\bar{\mathbb{Q}}(x)^\times$ such that $r\frac{\sigma(v)}{v}=ex^n\frac{p_1}{p_2}$ is \emph{reduced}, with $e\in\bar{\mathbb{Q}}^\times$ such that $e\in q^\mathbb{Z}$ if and only if $e=1$, $n\in\mathbb{Z}$, and $p_1,p_2\in\bar{\mathbb{Q}}[x]$ monic such that $\mathrm{gcd}(x,p_1)=\mathrm{gcd}(x,p_2)=\mathrm{gcd}(p_1,\sigma^m(p_2))$ for every $m\in\mathbb{Z}$, we would then have that the reduced form of $\sigma(r)$ is exactly $\sigma(r)\frac{\sigma^2(v)}{\sigma(v)}=eq^nx^n\frac{\sigma(p_1)}{\sigma(p_2)}$, and therefore \[\frac{\sigma(r)}{r}\frac{\sigma^2(v)}{v}=q^n\frac{\sigma(p_1)p_2}{p_1\sigma(p_2)}.\] This element is not necessarily reduced with respect to $\sigma^2$, but the reduced form of $\frac{\sigma(r)}{r}$ with respect to $\sigma^2$ is given by \[q^\varepsilon\frac{\tilde{p}_1}{\tilde{p}_2},\] where $\varepsilon=0$ if $n$ is even and $\varepsilon=1$ if $n$ is odd, and $\tilde{p}_1,\tilde{p}_2\in\bar{\mathbb{Q}}[x]$ are again monic such that $\mathrm{gcd}(x,\tilde{p}_1)=\mathrm{gcd}(x,\tilde{p}_2)=\mathrm{gcd}(\tilde{p}_1,\sigma^{2m}(\tilde{p}_2))=1$ for every $m\in\mathbb{Z}$. We then have that the reduced form of $\frac{\sigma(r)}{r}$ with respect to $\sigma^2$ is $\left(\frac{\tilde{p}_1}{\tilde{p}_2}\right)^2$, and therefore $\left(\frac{\sigma(r)}{r}\right)^2=\frac{\sigma^2(f)}{f}$ for some $f\in\bar{\mathbb{Q}}(x)^\times$ if and only if $\tilde{p}_1=1=\tilde{p}_2$, but this would imply that $\frac{\sigma(r)}{r}=\frac{\sigma^2(\tilde{f})}{\tilde{f}}$ for some $\tilde{f}\in\bar{\mathbb{Q}}(x)^\times$ already, which in turn would imply that $\alpha_1=\alpha_2$ for every $\left(\begin{smallmatrix}\alpha_1 & 0 \\ 0 & \alpha_2\end{smallmatrix}\right)\in\Delta(H_1)$, which is not possible according to the classification of Lemma~\ref{imprimitive-classification}.

In fact, we may pursue this further to conclude that it is also impossible to have \begin{equation}\label{imprimitive-2-klein}\frac{\delta(\sigma(r))}{\sigma(r)}-\frac{\delta(r)}{r}=\sigma^2(g)-g+c\end{equation} for some $g\in\bar{\mathbb{Q}}(x)$ and $c\in\mathbb{Z}$. This is because if, say, $\tilde{p}_1\neq 1$, then there would exist $\beta\in\bar{\mathbb{Q}}^\times$ such that $\tilde{p}_1(\beta)=0$, and then we would have that \[q^2\text{-}\mathrm{dres}\left(\frac{\delta(\tilde{p}_1)}{\tilde{p}_1},[\beta]_{q^2},1\right)\neq 0 = q^2\text{-}\mathrm{dres}\left(\frac{\delta(\tilde{p}_2)}{\tilde{p}_2},[\beta]_{q^2},1\right),\] and similarly with the roles of $\tilde{p}_1$ and $\tilde{p}_2$ exchanged. But since \[\frac{\delta(\sigma(r))}{\sigma(r)}-\frac{\delta(r)}{r}=\frac{\delta(\tilde{p}_1)}{\tilde{p}_1}-\frac{\delta(\tilde{p}_2)}{\tilde{p}_2}\] modulo $(\sigma^2-1)(\bar{\mathbb{Q}}(x))$, we see that \eqref{imprimitive-2-klein} is impossible unless $\tilde{p}_1=1=\tilde{p}_2$, which we already ruled out above.
\end{rem}

Having computed the $\sigma$-Galois group $H_1$ for \eqref{mateq} over $k_1$ as above, we can now compute the $\sigma\delta$-Galois group $G_1$ for \eqref{mateq} over $k_1$ with the following result.

\begin{prop}\label{imprimitive-prop-2} Suppose there are no solutions to \eqref{ric1} in $\bar{\mathbb{Q}}(x_2)$, and either $a=0$ or there exists a solution to \eqref{ric2} in $\bar{\mathbb{Q}}(x)$. Then $H_1\neq\{\pm1\}^2\ltimes\mathbb{G}_m(C)$ as in Lemma~\ref{imprimitive-classification}(3), and $G_1$ is the subgroup of \begin{equation}\label{std-imprimitive-gp} \{\pm1\}\ltimes\mathbb{G}_m(C)^2=\left\{\begin{pmatrix}\alpha _1& 0 \\ 0 & \alpha_2\end{pmatrix}, \begin{pmatrix}0 & \lambda_1 \\ \lambda_2 & 0 \end{pmatrix} \ \middle| \ \alpha_1\alpha_2\neq 0, \ \lambda_1\lambda_2\neq 0\right\}\end{equation} defined by the following conditions on $\alpha_1,\alpha_2,\lambda_1$, and $\lambda_2$.
\begin{enumerate} 
\item If $H_1=D_m^-$ as in Lemma~\ref{imprimitive-classification}(1) or $H_1=D_m^+$ as in Lemma~\ref{imprimitive-classification}(2), then $G_1=H_1$.
\item If $H_1=\{\pm1\}\ltimes\mathbb{G}_m(C)^2$ as in Lemma~\ref{imprimitive-classification}(4), then:
\begin{enumerate}
\item there exist $c\in \mathbb{Z}$ and $g\in\bar{\mathbb{Q}}(x)$ such that $\frac{\delta(b)}{b}=\sigma(g)-g+c$ if and only if $\delta\left(\frac{\delta(\alpha_1)}{\alpha_1}+\frac{\delta(\alpha_2)}{\alpha_2}\right)=0=\delta\left(\frac{\delta(\lambda_1)}{\lambda_1}+\frac{\delta(\lambda_2)}{\lambda_2}\right)$; moreover, $c=0$ if and only if $\delta(\alpha_1\alpha_2)=0=\delta(\lambda_1\lambda_2)$;
\item otherwise, $G_1=H_1$.
\end{enumerate}
\end{enumerate}\end{prop}

\begin{proof} The fact that $H_1\neq\{\pm1\}^2\ltimes\mathbb{G}_m(C)$ as in Lemma~\ref{imprimitive-classification}(3) under these conditions was already established in Remark~\ref{imprimitive-2-h-rem}. Let us denote by $\Delta(G_1)$ the subgroup of diagonal matrices in $G_1$, which coincides with the $\sigma^2\delta$-Galois group for \eqref{imprimitive-2-diag-system} over $k_1$. We may compute $\Delta(G_1)$ using the results of Proposition~\ref{galdiag} and Remark~\ref{diagonalizable-rem}. We again denote by $M\subseteq\mathbb{Z}^2$ the submodule generated by $(m_1,m_2)\in\mathbb{Z}^2$, not both zero and with $\mathrm{gcd}(m_1,m_2)=1$, such that there exist $c\in\mathbb{Z}$ and $g\in\bar{\mathbb{Q}}(x)$ such that \begin{equation}\label{imprimitive-2-symmetry}m_1\frac{\delta(\sigma(r))}{\sigma(r)}+m_2\frac{\delta(r)}{r}=\sigma^2(g)-g+c,\end{equation} which is equivalent to \[m_2\frac{\delta(\sigma(r))}{\sigma(r)}+m_1\frac{\delta(r)}{r}=\sigma^2\left(\sigma(g)-m_1\frac{\delta(r)}{r}\right)-\left(\sigma(g)-m_1\frac{\delta(r)}{r}\right)+c.\] As we saw in Remark~\ref{diagonalizable-rem}, either $M=\{0\}$ is trivial; or $M=\mathbb{Z}\cdot(m_1,m_2)$; or $M=\mathbb{Z}^2$. But it follows from the above computation that if $M$ is not trivial, then at least one of $(1,1)$ or $(1,-1$ belongs to $M$. But we saw in Remark~\ref{imprimitive-2-h-rem} that we cannot have $(1,-1)\in M$, since the relation~\ref{imprimitive-2-klein} is impossible. The only possibilities that remain are therefore $M=\{0\}$ or $M=\mathbb{Z}\cdot (1,1)$.

If $M=\{0\}$, then $\Delta(G_1)=\mathbb{G}_m(C)^2$, and therefore $G_1=H_1=\{\pm1\}\ltimes\mathbb{G}_m(C)^2$. Let us now suppose that $M=\mathbb{Z}\cdot(1,1)$. Then if $H_1=D_m^-$ as in Lemma~\ref{imprimitive-classification}(1) or $H_1=D_m^+$ as in Lemma~\ref{imprimitive-classification}(2), then $\Delta(G_1)=\Delta(H_1)$, which implies that $G_1=H_1$, as claimed in item~(1). It remains to establish item~(2) under the assumption that $M=\mathbb{Z}\cdot(1,1)$. But here we again have that $\mathrm{det}(\Delta(G_1))\subseteq\left\{\alpha\in C^\times \ \middle| \  \delta\left(\frac{\delta(\alpha)}{\alpha}\right)\right\}$ with equality if and only if $c\neq 0$ in \eqref{imprimitive-2-symmetry}, and moreover this $c=0$ if and only if $\mathrm{det}(\Delta(G_1))=\{\alpha\in C^\times \ | \ \delta(\alpha)=0\}$. Since $\Delta(G_1)$ has finite index in $G_1$ and $\mathrm{det}(\Delta(G_1))$ is divisible in either case, we obtain that $\mathrm{det}(G_1)=\mathrm{det}(\Delta(G_1))$, which concludes the proof of item~(2).
\end{proof}

\subsection{Irreducible and imprimitive (3): quadratic system of imprimitivity}\label{imprimitive-sec:3}

Supposing there are no solutions to \eqref{ric1} in $\bar{\mathbb{Q}}(x_2)$, $a\neq 0$, and there are no solutions to \eqref{ric2} in $\bar{\mathbb{Q}}(x)$, let us now assume that there is a solution $e\in \bar{\mathbb{Q}}(x_2)$ to \eqref{ric2}, and therefore the Galois conjugate $\bar{e}$ of $e$ over $\bar{\mathbb{Q}}(x)$ also satisfies \eqref{ric2}, since this Riccati equation with respect to $\sigma^2$ is defined over $\bar{\mathbb{Q}}(x)$. Here again we have that the non-existence of solutions to \eqref{ric1} in $\bar{\mathbb{Q}}(x_2)$ implies that there are no solutions to \eqref{ric1} in all of $k_\infty$, which implies that $H_\infty$ is irreducible as explained in \S\ref{hendriks-sec}. Since $H_\infty\subseteq H_2\subseteq H_1$ (which again denote the $\sigma$-Galois groups for \eqref{mateq} over $k_\infty$, $k_2$, and $k_1$, respectively), we then have that $H_2$ and $H_1$ must also be irreducible. Moreover the existence of the solution $e\in k_2$ to \eqref{ric2} implies that $H_\infty$ must be imprimitive, and since $H_\infty$ has finite index in $H_1$ and in $H_2$, the classification of the algebraic subgroups of $\mathrm{GL}_2(C)$ of \cite{gl2-classification} implies that $H_1$ and $H_2$ must be imprimitive also. By \cite[Prop.~12.2(1)]{vanderput-singer:1997}, both $H_1$ and $H_2$ must have bicyclic groups of connected components, and thus they must both be included in the list of irreducible imprimitive subgroups given in Lemma~\ref{imprimitive-classification}. By Corollary~\ref{fiber-cor-1}, $H_2\subseteq H_1$ has index either $1$ or $2$. We will show that $H_2\neq H_1$, which implies that the index of $H_2$ in $H_1$ is exactly $2$. A straightforward computation shows that the only groups listed in Lemma~\ref{imprimitive-classification} admitting another such group as an index-$2$ subgroup are $H_1=D_m^+$ with $m$ even, with $H_2$ then given by one of the groups $D_{m/2}^-$ or $D_{m/2}^+$.

To see that $H_2\neq H_1$ in this case, recall from \cite[Thm.~18]{hendriks:1997} that $e\in \bar{\mathbb{Q}}(x_2)$ satisfies \eqref{ric2} if and only if $d:=e+\frac{b}{a}$ has the property that $dy+\sigma(y)=:z_d$ satisfies $\sigma^2(z_d)+rz_d=0$ with $r:=-a\sigma(a)+\sigma(b)+a\sigma^2(d)$ if and only if $y$ satisfies \eqref{difeq}. We see that this is equivalent to $z_{\bar{d}}:=\bar{d}y+\sigma(y)$ satisfying $\sigma^2(z_{\bar{d}})+\bar{r}z_{\bar{d}}=0$, where $\bar{d}$ and $\bar{r}$ denote the Galois conjugates of $d,r\in \bar{\mathbb{Q}}(x_2)$ over $\bar{\mathbb{Q}}(x)$. Since $e\neq\bar{e}$, we also have $d\neq\bar{d}$ and $r\neq\bar{r}$. At this point, we could compute $H_2$ directly as in Remark~\ref{imprimitive-2-h-rem}, where in particular the subgroup of diagonal matrices $\Delta(H_2)$ in $H_2$ corresponds to the $\sigma^2$-Galois group for the system \begin{equation}\label{imprimitive-3-diag-system}\sigma^2(Y^{(2)})=A^{(2)}Y^{(2)}\end{equation} over $k_2$, where $A^{(2)}:=\sigma(A)A$. A computation shows that setting \[T:=\begin{pmatrix}d & 1 \\ \bar{d} & 1\end{pmatrix}\in\mathrm{GL}_2(\bar{\mathbb{Q}}(x_2))\] we have that \[\sigma^2(T)A^{(2)}T^{-1}=\begin{pmatrix} -r & 0 \\ 0 & -\bar{r}\end{pmatrix},\] and therefore \eqref{imprimitive-3-diag-system} is equivalent over $k_2$ to the system \begin{equation}\label{imprimitive-3-diag-system-z} \sigma^2(Z^{(2)})=\begin{pmatrix}-r & 0 \\ 0 & -\bar{r}\end{pmatrix} Z^{(2)}\end{equation} via the gauge transformation $Z^{(2)}=TY^{(2)}$.

If, contrary to our contention, we did have that $H_1=H_2$, then the $\sigma^2$-Galois group $H_1^{(2)}$ for the system \eqref{imprimitive-3-diag-system} over $k_1$ would coincide with $\Delta(H_2)$, and in particular we would have $H_1^{(2)}=\Delta(H_1)=\Delta(H_2)$ being diagonal. We will show that this is not the case. For this, consider the system \begin{equation}\label{imprimitive-3-big}\sigma^2(W)=\begin{pmatrix}A^{(2)} & 0 \\ 0 & q\end{pmatrix}W, \quad\text{with fundamental solution matrix}\quad W=\begin{pmatrix} Y^{(2)} & 0 \\ 0 & x_2\end{pmatrix},\end{equation} where $Y^{(2)}$ in turn denotes a $(2\times 2)$ fundamental solution matrix for \eqref{imprimitive-3-diag-system} over $k_1$. Let $\tilde{H}^{(2)}$ denote the $\sigma^2$-Galois group for the system \eqref{imprimitive-3-big} over $k_1$. Let $\tilde{\tau}\in\tilde{H}^{(2)}$ such that $\tilde{\tau}(x_2)=-x_2$, and let $\tau:=\tilde{\tau}|_{S_1}\in H_1^{(2)}$ denote the restriction of $\tilde{\tau}$ to the $\sigma^2$-PV ring  corresponding to the system \eqref{imprimitive-3-diag-system}: $S_1^{(2)}:=k_1[Y^{(2)},\mathrm{det}(Y^{(2)})^{-1}]$. Let $M_\tau\in\mathrm{GL}_2(C)$ denote the matrix correspondiong to $\tau\in H_1^{(2)}$, so that $\tau(Y^{(2)})=Y^{(2)}M_\tau$. Since the system \eqref{imprimitive-3-diag-system-z} is diagonal, we have that \[TY^{(2)}=Z^{(2)}=\begin{pmatrix} z_1 & 0 \\ 0 & z_2\end{pmatrix},\] where $\sigma^2(z_1)=-rz_1$ and $\sigma^2(z_2)=-\bar{r}z_2$. But then we see that \[\tilde{\tau}(Z^{(2)})=\tilde{\tau}(TY^{(2)})=\bar{T}Y^{(2)}M_\tau=\begin{pmatrix}0 & 1 \\ 1 & 0 \end{pmatrix}Z^{(2)}M_\tau.\] On the other hand, $\sigma^2(\tilde{\tau}(z_1))=\tilde{\tau}(\sigma^2(z_1))=\tilde{\tau}(-rz_1)=-\bar{r}\tilde{\tau}(z_1)$, and therefore $\tilde{\tau}(z_1)=\lambda_2z_2$ for some $\lambda_2\in C^\times$. Similarly we see that $\tilde{\tau}(z_2)=\lambda_1z_1$ for some $\lambda_1\in C^\times$, and therefore \[\tilde{\tau}(Z^{(2)})=\begin{pmatrix}0 & 1 \\ 1 & 0 \end{pmatrix}Z^{(2)}\begin{pmatrix} 0 & \lambda_1 \\ \lambda_2 & 0 \end{pmatrix}.\] This shows that $M_\tau=\left(\begin{smallmatrix} 0 & \lambda_1\\ \lambda_2 & 0 \end{smallmatrix}\right)$, as we wanted to show.

\begin{prop}\label{imprimitive-prop-3} Suppose there are no solutions to \eqref{ric1} in $\bar{\mathbb{Q}}(x_2)$, $a\neq 0$, and there are no solutions to \eqref{ric2} in $\bar{\mathbb{Q}}(x)$ but there exists a solution to \eqref{ric2} in $\bar{\mathbb{Q}}(x_2)$. Then $G_1=H_1=D_m^+$ for the smallest even positive integer $m\in 2\mathbb{N}$ such that $b^m=\frac{\sigma(f)}{f}$ for some $f\in\bar{\mathbb{Q}}(x)^\times$.\end{prop}

\begin{proof} The remarks above show that under these assumptions the $\sigma$-Galois group $H_2$ for \eqref{mateq} over $k_1$ has index exactly $2$ in the $\sigma$-Galois group $H_1$ for \eqref{mateq} over $k_1$. Thus $H_1=D_m^+$ as in Lemma~\ref{imprimitive-classification}(2) for some even positive integer $m\in2\mathbb{N}$, and $H_2$ is then one of $D_{m/2}^-$ or $D_{m/2}^+$. In either case, it follows from Proposition~\ref{imprimitive-2-prop}, applied over $k_2$ instead of $k_1$, that $H_2=G_2$ is also the $\sigma\delta$-Galois group $G_2$ for \eqref{mateq} over $k_2$. Since the index of $G_2$ in $G_1$, the $\sigma\delta$-Galois group for \eqref{mateq} over $k_1$, is also $2=[H_1:H_2]$, in then follows that $G_1=H_1=D_m^+$ in this case, as claimed.
\end{proof}


\section{Irreducible and primitive groups} \label{large-sec}

Let us denote again $k=C(x)$, where $C$ is a $\delta$-closure of $\bar{\mathbb{Q}}$, $\sigma$ denotes the $C$-linear automorphism of $k$ defined by $\sigma(x)=qx$, and $\delta(x)=1$. We write $H$ for the $\sigma$-Galois group and $G$ for the $\sigma\delta$-Galois group for \begin{equation}\label{difeqsl2}\sigma(Y)=\begin{pmatrix} 0 & 1 \\ -b & -a\end{pmatrix}Y\end{equation} over $k$, where $a,b\in\bar{\mathbb{Q}}(x)$ and $b\neq 0$. In this section we consider the case where $a\neq 0$ and there are no solutions in $\bar{\mathbb{Q}}(x_2)$ to \eqref{ric1} nor to \eqref{ric2}, which is equivalent to the condition that $\mathrm{SL}_2(C)\subseteq H$ by the results of \cite{hendriks:1997} summarized in \S\ref{hendriks-sec}. In this case, $H$ is reductive and the connected component of the identity $H^\circ$ is either $\mathrm{SL}_2(C)$ or $\mathrm{GL}_2(C)$, and in either case the derived subgroup $H^{\circ,\mathrm{der}}=\mathrm{SL}_2(C)$. Therefore by \cite[Thm~.5.2]{arreche-singer:2016} $\mathrm{SL}_2(C)\subseteq G$, and hence $G\subseteq\mathrm{GL}_2(C)$ is determined by the image the determinant map $\mathrm{det}(G)\subseteq\mathbb{G}_m(C)$, which is the $\sigma\delta$-Galois group for $\sigma(y)=by$ over $k$. The proof of the following result is immediate.

\begin{prop}\label{sl2-prop} Suppose there are no solutions to \eqref{ric1} in $\bar{\mathbb{Q}}(x_2)$, $a\neq 0$, and there are no solutions to \eqref{ric2} in $\bar{\mathbb{Q}}(x_2)$. Then $\mathrm{det}(G)\subseteq\mathbb{G}_m(C)$ is determined as follows.
\begin{enumerate}
\item There exist a smallest positive integer $m\in\mathbb{N}$ and $f\in\bar{\mathbb{Q}}(x)^\times$ such that $b^m=\frac{\sigma(f)}{f}$ if and only if $\mathrm{det}(G)=\mu_m$, the group of $m$-th roots of unity.
\item There exist $0\neq c\in\mathbb{Z}$ and $f\in\bar{\mathbb{Q}}(x)$ such that $\frac{\delta(b)}{b}=\sigma(f)-f+c$ if and only if $\mathrm{det}(G)=\left\{\alpha\in C^\times \ \middle| \ \delta\left(\frac{\delta(\alpha)}{\alpha}\right)=0\right\}$.
\item There exists $f\in\bar{\mathbb{Q}}(x)$ such that $\frac{\delta(b)}{b}=\sigma(f)-f$ if and only if $\mathrm{det}(G)=\{\alpha\in C^\times \ | \ \delta(\alpha)=0\}$.
\item Otherwise, $\mathrm{det}(G)=\mathbb{G}_m(C)$.
\end{enumerate}
\end{prop}


\section{Examples} \label{examples-sec}
In this section we compute the $\sigma\delta$-Galois group $G$ associated to some concrete second-order linear difference equations over $\bar{\mathbb{Q}}(x)$ with respect to the $q$-dilation operator $\sigma(x)=qx$, where $q\in\mathbb{C}^\times$ is not a root of unity. We will first apply the algorithm of \cite{hendriks:1997} to compute the $\sigma$-Galois group $H$ associated to the equation, and then apply the procedures developed in this paper to compute $G$.

\subsection{Example} Let us consider \eqref{difeq} with \begin{align*} b&=q^3x^6\frac{(x-1)^4(q^2x^2+6qx+6)}{x^2+6x+6} ;\quad\text{and}\\
a &=-q^3x^3\frac{(2q^2)x^4+4(q^2+q)x^3+(7q^2-24q+7)x^2-6(q+1)x+12}{x^2+6x+6}.\end{align*}

Applying the procedure in \cite[\S4.1]{hendriks:1997} or using a computer algebra system (for example, with the QHypergeometricSolution command included in the Maple package QDifferenceEquations) one can verify that there is exactly one solution $u\in\bar{\mathbb{Q}}(x)$ to the first Riccati equation \eqref{ric1} in this case, given by \[u:=x^3(x-1)^2.\] After computing \[\frac{b}{u\sigma(u)}=\frac{(x-1)^2(q^2x^2+6qx+6)}{(qx-1)^2(x^2+6x+6)},\] we see that \[w:=\frac{x^2+6x+6}{(x-1)^2}\in\bar{\mathbb{Q}}(x)^\times\] satisfies $\sigma(w)=\frac{b}{u\sigma(u)}w$, and therefore we are in the setting of Proposition~\ref{wrat}. After verifying that \[\frac{\delta(u)}{u}=5+\frac{2}{x-1}\neq\sigma(f)-f+c\qquad\text{for any} \quad f\in\bar{\mathbb{Q}}(x) \ \text{and} \ c\in\mathbb{Z},\] we proceed to attempt to find a linear differential operator $\mathcal{L}\in\bar{\mathbb{Q}}[\delta]$ of smallest possible order such that there exist $f\in\bar{\mathbb{Q}}(x)$ and $c\in\bar{\mathbb{Q}}$ satisfying \[\mathcal{L}\left(\frac{\delta(u)}{u}\right)-w=\sigma(f)-f+c.\] Since \[w=\frac{x^2+6x+6}{(x-1)^2}=\frac{5}{(x-1)^2}+\frac{8}{x-1}+1\] has as its only non-zero $q$-discrete residues:\[q\text{-}\mathrm{dres}(w,[1]_q,2)=5;\qquad q\text{-}\mathrm{dres}(w,[1]_q,1)=8;\qquad\text{and}\quad q\text{-}\mathrm{dres}(w,\infty)=1,\] we see that if there exists such an $\mathcal{L}\in\bar{\mathbb{Q}}[\delta]$ then its order must be exactly $1$. Writing $\mathcal{L}=e_1\delta+e_0$, we find that \[\mathcal{L}\left(\frac{\delta(u)}{u}\right)-w=\frac{-2e_1-5}{(x-1)^2}+\frac{-2e_1+2e_0-8}{x-1}+5e_0-1,\] which has the desired form $\sigma(f)-f+c$ for some $f\in\bar{\mathbb{Q}}(x)$ and $c\in\bar{\mathbb{Q}}$ if and only if \[-2e_1-5=0=-2e_1+2e_0-8 \quad\Longleftrightarrow\quad e_1=-\tfrac{5}{2} \ \text{and} \ e_0=\tfrac{3}{2}.\] The corresponding value of $c=5e_0-1=\frac{13}{2}\neq 0$. With this, we conclude that the $\sigma\delta$-Galois group for \eqref{difeq} over $k_1$ for this choice of coefficients $a,b\in\bar{\mathbb{Q}}(x)$ is \[G=\left\{\begin{pmatrix}\alpha & \xi \\ 0 & \alpha\end{pmatrix} \ \middle| \ \alpha,\xi\in C, \ \alpha\neq 0, \ \delta\left(\frac{\xi}{\alpha}\right)=-\frac{5}{2}\delta^2\left(\frac{\delta(\alpha)}{\alpha}\right)+\frac{3}{2}\delta\left(\frac{\delta(\alpha)}{\alpha}\right)\right\}.\]

\subsection{Example} Let us consider \eqref{difeq} with \[a=-(q+q^{1/2})x\qquad\text{and}\qquad b=q^{1/2}(x^2-x).\] Since the valuations at $x=0$ of the coefficients are $v(a)=1$ and $v(b)=1$, we are in the case where $v(b)\leq 2v(a)$ and $v(b)$ is odd, and therefore there are no solutions to \eqref{ric1} in $\bar{\mathbb{Q}}(x)$ (cf.~\cite[\S4.1]{hendriks:1997}). However, $u=x+x^{1/2}\in\bar{\mathbb{Q}}(x^{1/2})$ and $\bar{u}=x-x^{1/2}\in\bar{\mathbb{Q}}(x^{1/2})$ both satisfy \eqref{ric1}. Since \[\frac{\delta(b)}{b}=2+\frac{1}{x-1}\neq\sigma(f)-f+c \qquad\text{for any} \quad f\in\bar{\mathbb{Q}}(x) \ \text{and} \ c\in\mathbb{Z},\] we deduce that $\mathrm{det}(G)=\mathbb{G}_m(C)$. Since there is no $f\in\bar{\mathbb{Q}}(x^{1/2})^\times$ such that \[\left(\frac{u}{\bar{u}}\right)^2=x\left(\frac{x^{1/2}+1}{x^{1/2}-1}\right)^2=\frac{\sigma(f)}{f},\] we conclude that the $\sigma\delta$-Galois group $G$ for \eqref{difeq} over $k_1$ for this choice of coefficients $a,b\in\bar{\mathbb{Q}}(x)$ is \[G=\{\pm1\}\ltimes\mathbb{G}_m(C)^2=\left\{\begin{pmatrix}\alpha_1 & 0 \\ 0 & \alpha_2\end{pmatrix},\begin{pmatrix} 0 & \lambda_1 \\ \lambda_2 & 0 \end{pmatrix} \ \middle| \ \alpha_1,\alpha_2,\lambda_1,\lambda_2\in C^\times\right\}.\]

\subsection{Example} Let us consider \eqref{difeq} with $a=0$ and $b=-q^{1/2}x$. This example was discussed in \cite[\S4.1]{arreche-singer:2016}, as an example of a projectively integrable system whose $\sigma$-Galois group $H$ was solvable but not abelian; in fact it was proved there using ad-hoc methods that \[H=\left\{\begin{pmatrix} \alpha & 0 \\ 0 & \alpha \end{pmatrix},\begin{pmatrix}\alpha & 0 \\ 0 & -\alpha\end{pmatrix},\begin{pmatrix}0 & \lambda \\ \lambda & 0\end{pmatrix},\begin{pmatrix} 0 & \lambda \\ -\lambda & 0 \end{pmatrix} \ \middle| \ \alpha,\lambda\in C^\times\right\}.\] We can now prove this systematically, as well as find the corresponding $\sigma\delta$-Galois group $G$, using the results of \S\ref{imprimitive-sec:1}. Since the valuations at $x=0$ of the coefficients are $v(a)=\infty$ and $v(b)=1$, we are in the case where $v(b)\leq 2v(a)$ and $v(b)$ is odd, and therefore there are no solutions to \eqref{ric1} in $\bar{\mathbb{Q}}(x)$ (cf.~\cite[\S4.1]{hendriks:1997}). However we see that $u=x^{1/2}\in\bar{\mathbb{Q}}(x^{1/2})$ and $\bar{u}=-u=-x^{1/2}\in\bar{\mathbb{Q}}(x^{1/2})$ both satisfy \eqref{ric1}. Since $\frac{\delta(b)}{b}=1,$ we see that $\mathrm{det}(G)=\left\{\alpha\in C^\times \ \middle| \ \delta\left(\frac{\delta(\alpha)}{\alpha}\right)=0\right\}$. We also verify that $\left(\frac{u}{\bar{u}}\right)=(-1)^2=1.$ This concludes the computation that \[G=\left\{\begin{pmatrix} \alpha & 0 \\ 0 & \alpha \end{pmatrix},\begin{pmatrix}\alpha & 0 \\ 0 & -\alpha\end{pmatrix},\begin{pmatrix}0 & \lambda \\ \lambda & 0\end{pmatrix},\begin{pmatrix} 0 & \lambda \\ -\lambda & 0 \end{pmatrix} \ \middle| \ \alpha,\lambda\in C^\times, \ \delta\left(\frac{\delta(\alpha)}{\alpha}\right)=0=\delta\left(\frac{\delta(\lambda)}{\lambda}\right)\right\}.\]

\subsection{Example} In \cite{KoutschanYi:2018} the authors develop algorithms for desingularization of $q$-difference--differential operators. In \cite[Example~5.2]{KoutschanYi:2018}, those results were applied in the study of the difference equations satisfied by the \emph{colored Jones polynomials} of several knots. In spite of the name, a colored Jones polynomial is not actually a polynomial in general, but rather consist of an infinite sequence of rational functions in $\mathbb{Q}(q)$, where $q$ is a formal indeterminate. We refer to \cite[\S5]{KoutschanYi:2018} and the references therein for additional details.

This second-order difference equation is satisfied by the colored Jones polynomial (after normalization) of the knot $K^{\mathrm{twist}}_{-1}$; we emphasize that the name ``polynomial" may be misleading: in general, the colored Jones polynomial of a knot actually consists of an infinite sequence of rational functions in $\mathbb{Q}(q)$. 
Let us consider~\eqref{difeq} with 
$$a= \frac{(q x-1) (q x+1) \bigl(q^4 x^4-q^3 x^3-q^3 x^2-q x^2-q x+1\bigr)}{q^2 x^2 \bigl(q x^2-1\bigr)}$$ and 
$$b= \frac{q^3 x^2-1}{q x^2-1}$$
\begin{equation} \label{EQ:JonesPoly}
y(q^2x)+a(x)y(qx)+b(x)y(x)=0.
\end{equation}
The corresponding second-order linear difference equation with this choice of coefficients $a,b\in\overline{\mathbb{Q}(q)}(x)$, where $q$ is a formal indeterminate, is satisfied by the colored jones polynomial for $K^{\mathrm{twist}}_{-1}$ (see \cite[\S5, Fig.~1]{KoutschanYi:2018}).

To compute the $\sigma\delta$-Galois group for this equation over $C(x)$, where $C$ is a $\delta$-closure of the $\delta$-constant field $\overline{\mathbb{Q}(q)}$, proceeds as follows. Using the QHypergeometricSolution command included in the Maple package QDifferenceEquations, we have verified that the Riccati equations \eqref{ric1} and \eqref{ric2} do not admit any solutions in $C(x^{1/2})$. Therefore $\mathrm{SL}_2(C)\subseteq G\subseteq H$, where $H$ denotes the $\sigma$-Galois group, as discussed in \S\ref{hendriks-sec} and \S\ref{large-sec}. We see that \[b= \frac{q^3 x^2-1}{q x^2-1}=\frac{\sigma(qx^2-1)}{qx^2-1}.\] Therefore, $G=\mathrm{SL}_2(C)$ in this case.



\end{document}